\documentclass[a4,11pt]{nmd/article}
\usepackage{enumitem}
\usepackage{graphicx,color}
\usepackage{epsfig,manfnt}
\usepackage{amsfonts}
\usepackage{amssymb}
\usepackage{amsmath}
\usepackage{amsthm}
\usepackage{latexsym}
\usepackage{amscd}
\usepackage{flafter}
\usepackage{epsf}
\usepackage{epstopdf}
\usepackage{inputenc}
\usepackage{stmaryrd}
\usepackage{tikz,tikz-cd}
\usepackage{mathrsfs}
\usepackage{lipsum}
\usepackage[square,sort,comma,numbers]{natbib}
\newcommand{\A}{{\mathbb{A}}}
\newcommand{\Acal}{{\mathcal{A}}}

\newcommand{\spinc}{\mathfrak{s}}
\newcommand{\spinct}{\mathfrak{t}}
\usepackage[small,nohug,heads=vee]{diagrams}
\diagramstyle[labelstyle=\scriptstyle]
\newcommand{\alphas}{\boldsymbol{\alpha}}
\newcommand{\betas}{\boldsymbol{\beta}}

\newcommand{\Hcal}{\mathcal{H}}

\newcommand{\pfrak}{\mathfrak{p}}
\newcommand{\SpinC}{{\mathrm{Spin}^c}}
\newcommand{\ra}{\rightarrow}

\newcommand*\circled[1]{\tikz[baseline=(char.base)]{
            \node[shape=circle,draw,inner sep=2pt] (char) {#1};}}

\title{Simple balanced three-manifolds,\\ Heegaard Floer homology and\\
 the Andrews-Curtis conjecture }
 \author{Neda Bagherifard}
\email{neda.bagherifard@gmail.com}
\author{Eaman Eftekhary}
\address{School of Mathematics, Institute for Research in 
	Fundamental Sciences (IPM), P. O. Box 19395-5746, Tehran, Iran}%
\email{eaman@ipm.ir}
\date{}

\begin{document}
\begin{abstract}
	The first author introduced a notion of equivalence on a family of $3$-manifolds with boundary, called (simple) balanced $3$-manifolds in an earlier paper and discussed the analogy between the Andrews-Curtis equivalence for group presentations and the aforementioned notion of equivalence.  Motivated by the Andrews-Curtis conjecture, we use  tools from Heegaard Floer theory to prove that there are simple balanced $3$-manifolds which are not in the trivial equivalence class (i.e. the equivalence class of $S^2\times [-1,1]$).
\end{abstract}
\maketitle
\tableofcontents
\newpage

\section{Introduction} 
Suppose that $R=\{b_1,\ldots,b_m\}$ is a finite subset of the free group $F(X)$ generated by the finite set $X=\{a_1,\ldots,a_n\}$. We denote by $(X|R)$ the quotient $G$ of $F(X)$ by the normal subgroup generated by $R$. The pair $(X,R)$ is then called a {\emph{presentation}} of $G$ with {\emph{generators}} $X$   and {\emph{relators}} $R$, which is balanced if $|X|=|R|$. An {\emph{extended Andrews-Curtis transformation}} (EAC-transformation for short) on  $(X,R)$   is defined as one of the following transformations, or its inverse, which of course results in another presentation of $G$ \cite{WR} (see also \cite{Hog-Met}):
\begin{itemize}
\setlength\itemsep{-0.3em}
	\item[1.] {\bf{Composition}}: Replace  $b\in R$ with $bb'$ for some $b'\neq b$ in $R$;
	\item[2.] {\bf{Inversion}}: Replace $b\in R$ with $b^{-1}$;
	\item[3.] {\bf{Cancellation}}: Replace $b=b'aa^{-1}b''\in R$  with $b'b''$, where $a\in X$ or $a^{-1}\in X$;
	\item[4.] {\bf{Stabilization}}: Add  a new element $a$ to both $X$ and $R$;
	\item[5.] {\bf{Replacement}}: Replace $a'a$ or $a'a^{-1}$ for $a'$ in all the relators for some $a\neq a'$ in $X$.
\end{itemize} 
{\emph{Stable Andrews-Curtis transformations}} (or SAC transformations) consist of the first $4$ transformations and their inverses. The presentations $P'=( X',R')$ and $P=(X,R)$ are called {\emph{EAC equivalent} (respectively, {\emph{SAC equivalent}}) if $P'$ is obtained from $P$ by a finite sequence of EAC transformations (respectively, SAC transformations). For the trivial group, the SAC equivalence class of a presentation is the same as its EAC equivalence class (\cite{WR}). The stable Andrews-Curtis conjecture (or SAC conjecture) states that every balanced presentation of the trivial group is SAC equivalent to the trivial presentation, i.e. $(a,a)$ (c.f. \cite{A-C}). Most experts expect that the SAC conjecture is not true and there are potential counterexamples (\citep{Brown, Bu-Mac, Miller, Shpil}). One of the simplest potential counterexamples for the SAC conjecture  is given by $P_0=(X_0,R_0)$, where
\begin{equation}\label{eq-AC-example}
 X_0=\big\{x,y\big\}\quad\text{and}\quad R_0=\big\{r=x^{-1}y^2xy^{-3},s=y^{-1}x^{2}yx^{-3}\big\},
\end{equation}
(see \citep{Shpil}). The group presentation $P_0$ is considered in this paper in correspondence with a notion of equivalence for balanced $3$-manifolds, as explained below. \\

A compact oriented $3$-manifold $N$ with boundary is called {\emph{balanced}} if each component of $N$ has two boundary components of the same genus.  Let $\partial^\pm N$ denote boundary components of $N$ where the orientation of $\partial^+N$ (resp. $\partial^-N$) matches with (resp. is the opposite of) the orientation inherited as the boundary of $N$. Let $\iota^\pm: \partial^\pm N\ra N$ denote the inclusion maps and $H^\pm$ denotes the normalizer of $\iota_*^\pm(\pi_1(\partial^\pm N))$ in $\pi_1(N)$. A balanced $3$-manifold is called {\emph{simple}} if for each connected component $N$ of it as above, both quotient groups $\pi_1(N)/H^\pm$  are trivial.  Associated with each Heegaard diagram $\mathcal{H}=(\Sigma,\alphas,\betas)$
of $N$, there are two balanced presentations $P_\alpha(\mathcal{H})$ and $P_\beta(\mathcal{H})$ for the latter quotient groups where for $P_\alpha(\mathcal{H})$ (resp. $P_\beta(\mathcal{H})$) the generators are in correspondence with the $\alphas$ (resp. $\betas$) and the relators are in correspondence with the $\betas$ (resp. $\alphas$)  (see \cite{NB}). Let $\pfrak_\alpha(N)$ and $\pfrak_\beta(N)$ denote the EAC equivalence classes of the presentations $P_\alpha(\mathcal{H})$ and $P_\beta(\mathcal{H})$, respectively. Note that these EAC equivalence classes are independent of the choice of the Heegaard diagram $\mathcal{H}$ for $N$. Similarly, we may define $\pfrak_\alpha(N)$ and $\pfrak_\beta(N)$ for a balanced $3$-manifold $N$ which is not connected. If $N$ is a simple balanced $3$-manifold, $\pfrak_\alpha(N)$ and $\pfrak_\beta(N)$ are both EAC equivalence classes of presentations for the trivial group.\\

A notion of equivalence in the family of balanced $3$-manifolds was introduced in \cite{NB}. We say that a balanced $3$-manifold $N$ simplifies to another balanced $3$-manifold $N'$ if there is an embedded cylinder $C\sim S^1\times [-1,1]$ in $N$,  with $\partial^\pm C\sim S^1\times\{\pm 1\}\subset \partial^\pm N$, such that $N'$ is obtained by cutting $N$ along $C$ and gluing two copies of $D^2\times [-1,1]$ to the resulting boundary cylinders in $N\setminus C$. We then write $N\xrightarrow{C} N'$. We say that a balanced $3$-manifold $N$ admits a {\emph{simplifier}} if there is a sequence of simplifications 
\[N=N_n\xrightarrow{C_n}N_{n-1}\xrightarrow{C_{n-1}}\cdots\xrightarrow{C_2}N_1\xrightarrow{C_1}N_0,\]
such that $N_0$ is a disjoint union of copies of $S^2\times [-1,1]$. The inverse of a simplification is called an {\emph{anti-simplification}}. Two balanced $3$-manifolds are called {\emph{equivalent}} if they may be changed to one-another by a finite sequence of simplifications, anti-simplifications and homeomorphisms. The equivalence of the balanced $3$-manifolds $N$ and $N'$ implies that $\pfrak_\alpha(N)=\pfrak_\alpha(N')$ and $\pfrak_\beta(N)=\pfrak_\beta(N')$.  Therefore, a pair of well-defined EAC equivalence classes (of group presentations) are assigned to each equivalence class of balanced $3$-manifolds and in this sense, the equivalence notion between balanced $3$-manifolds is {\emph{weaker}} than the EAC equivalence for group presentations. In the family of simple balanced $3$-manifolds, both EAC equivalence classes are presentations of the trivial group. Motivated by the SAC conjecture, it is thus natural to ask if there is a simple balanced $3$-manifold $N$ which is not equivalent to the trivial simple balanced $3$-manifold $S^2\times [-1,1]$? In this paper, we combine the main result of \cite{NB} with tools from Heegaard Floer theory (see \cite{OS}) to prove the following theorem.

\begin{theorem}\label{thm3}
	There is a simple balanced $3$-manifold $N$ with
	\[\pfrak_\alpha(N)=\pfrak_\beta(N)=P_0=[(X_0,R_0)],\]
where $P_0$ is given in (\ref{eq-AC-example}),	which is not equivalent to $S^2\times I$.
\end{theorem}

As mentioned above, besides Heegaard Floer theory, the main tool used in proving Theorem~\ref{thm3} is a fundamental result about the equivalence class of the simple balanced $3$-manifold $S^2\times [-1,1]$, which is proved in \cite{NB} and  may be stated as follows.

\begin{theorem}\cite[Theorem~1.6]{NB}\label{thm1}
	Every balanced $3$-manifold $N$  which is equivalent to $S^2\times I$ admits a simplifier.
\end{theorem} 

\begin{figure}
\def\svgwidth{0.85\textwidth}
{\scriptsize{
\begin{center}
\begingroup%
  \makeatletter%
  \providecommand\color[2][]{%
    \errmessage{(Inkscape) Color is used for the text in Inkscape, but the package 'color.sty' is not loaded}%
    \renewcommand\color[2][]{}%
  }%
  \providecommand\transparent[1]{%
    \errmessage{(Inkscape) Transparency is used (non-zero) for the text in Inkscape, but the package 'transparent.sty' is not loaded}%
    \renewcommand\transparent[1]{}%
  }%
  \providecommand\rotatebox[2]{#2}%
  \newcommand*\fsize{\dimexpr\f@size pt\relax}%
  \newcommand*\lineheight[1]{\fontsize{\fsize}{#1\fsize}\selectfont}%
  \ifx\svgwidth\undefined%
    \setlength{\unitlength}{444.46465743bp}%
    \ifx\svgscale\undefined%
      \relax%
    \else%
      \setlength{\unitlength}{\unitlength * \real{\svgscale}}%
    \fi%
  \else%
    \setlength{\unitlength}{\svgwidth}%
  \fi%
  \global\let\svgwidth\undefined%
  \global\let\svgscale\undefined%
  \makeatother%
  \begin{picture}(1,0.3666479)%
    \lineheight{1}%
    \setlength\tabcolsep{0pt}%
    \put(0,0){\includegraphics[width=\unitlength,page=1]{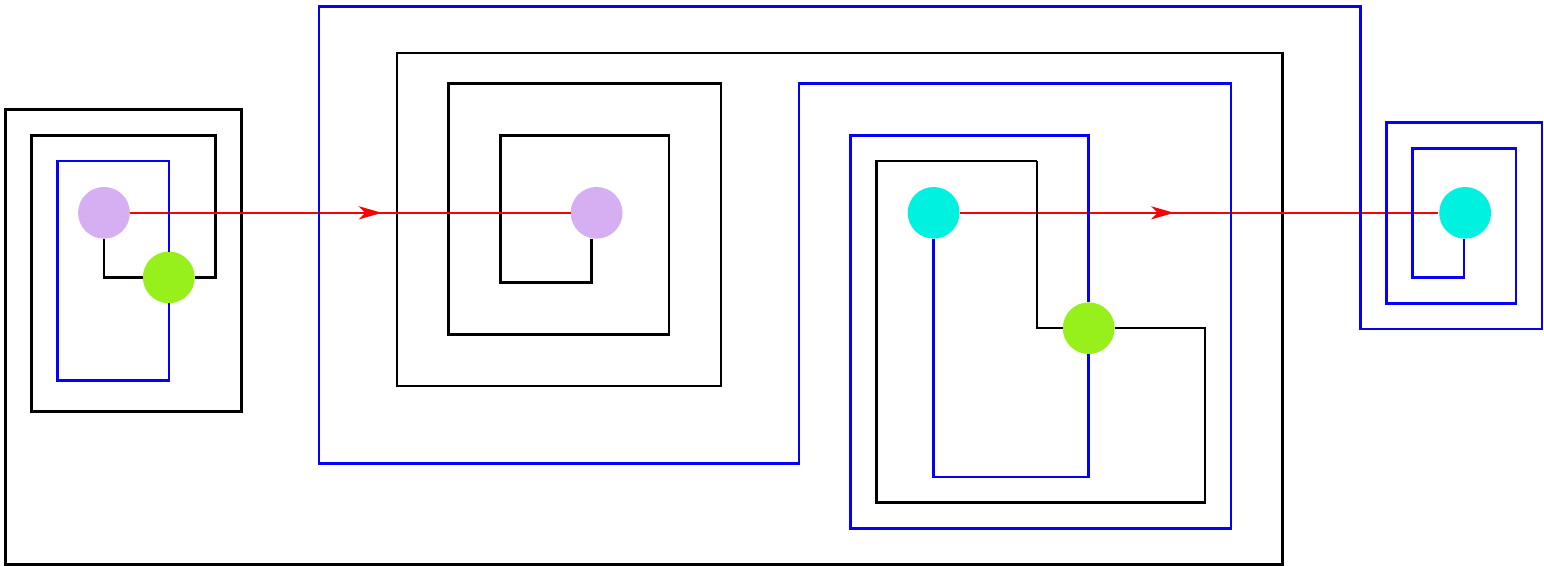}}%
    \put(0.0412986,0.1296078){\color[rgb]{0,0,0}\makebox(0,0)[lt]{\lineheight{0}\smash{\begin{tabular}[t]{l}$\beta_2$\end{tabular}}}}%
    \put(0.16785427,0.23591545){\color[rgb]{0,0,0}\makebox(0,0)[lt]{\lineheight{0}\smash{\begin{tabular}[t]{l}$\alpha_1$\end{tabular}}}}%
    \put(0.72807855,0.234228){\color[rgb]{0,0,0}\makebox(0,0)[lt]{\lineheight{0}\smash{\begin{tabular}[t]{l}$\alpha_2$\end{tabular}}}}%
    \put(0.00923755,0.00980086){\color[rgb]{0,0,0}\makebox(0,0)[lt]{\lineheight{0}\smash{\begin{tabular}[t]{l}$\beta_1$\end{tabular}}}}%
    \put(0,0){\includegraphics[width=\unitlength,page=2]{HD-1.pdf}}%
  \end{picture}%
\endgroup%

\caption{\label{ex-1} The Heegaard surface is a surface of genus three which is obtained by identifying the boundaries of disks with the same color. The curves are oriented in a way that the balanced presentation associated with this Heegaard diagram is $P_0$.}
\end{center}}}
\end{figure}

The group presentation $P_0$ of Equation~\ref{eq-AC-example} is realized by the Heegaard diagram 
\[\overline{\Hcal}=(\bar{\Sigma},\boldsymbol{\bar{\alpha}}=\{\alpha_1,\alpha_2\},\boldsymbol{\bar{\beta}}=\{\beta_1,\beta_2\}),\]
illustrated in Figure \ref{ex-1}. In fact, the Heegaard diagram $\overline{\Hcal}$ determines a simple balanced $3$-manifold $N$ with $\mathfrak{p}_\alpha(N)=\mathfrak{p}_\beta(N)=[P_0]$. If $N$ is equivalent to $S^2\times I$, Theorem \ref{thm1} implies that $N$ admits a  simplifier. We have $\partial N=\partial^+N\amalg -\partial^- N$ where $\partial^\pm N$ are surfaces of genus $1$. If $N$ admits a simplifier, there is a nontrivial cylinder $C$ in $N$ such that $\partial^\pm C$ in $\partial^\pm N$  are essential curves. Let $f:\partial^+N\rightarrow\partial^-N$  be  the homeomorphism from $\partial^+N$ to $\partial^-N$ which makes the following diagram commutative:
\begin{diagram}
H_1(\partial^+ N,\Z)& &\rTo{f_*}&&H_1(\partial^-N,\Z) \\
&\rdTo{\iota^+_*}&&\ldTo{\iota^-_*}&\\
&&H_1(N,\Z)&&
\end{diagram}
This criteria determines $f$ upto isotopy. Since $\partial^+C$ is homologous to $\partial^-C$, we may further assume that $f$ maps $\partial^+C$ to $\partial^-C$. Let $N_f$ denote the closed $3$-manifold obtained from $N$ by identifying $\partial^+N$ with $\partial^-N$ using $f$. Let $\bar{C}$ denote the torus in $M$ which is obtained from $C$ by identifying $\partial^+C$ with $\partial^-C$. Thus $\bar{C}$ and $\partial^+N\sim_f\partial^-N$  represent linearly independent  homology classes in $H_2(N_f,\Z)=\Z\oplus\Z\oplus \Z$ with zero Thurston semi-norm. Recall that the Thurston semi-norm of a closed $3$-manifold $M$ is defined on $H_2(M,\Z)$ by
\[\Theta:H_2(M,\Z)\rightarrow\Z^{\geq0},\quad\quad \Theta(\xi):=\min\left\{\chi_+(\Sigma)\ \big|\ \Sigma\hookrightarrow M\ \ \text{and}\ \ [\Sigma]=\xi\right\},\]
where the minimum is taken over all compact, oriented surfaces $\Sigma=\amalg_i\Sigma_i$ embedded in $M$ and representing the homology class $\xi$, while $\chi_+(\Sigma)$ is defined by $\sum_{\substack{g(\Sigma_i)>0}}(2g(\Sigma_i)-2)$ (see \citep{Thurston}). Heegaard Floer homology groups with twisted coefficients detect the Thurston semi-norm. More precisely, for a closed $3$-manifold $M$, let $\underline{\widehat{HF}}(M)$ denote the Heegaard Floer homology group of $M$ with twisted coefficients, which is a  $\Z/2\Z$-graded $\Z_2[H^1(M,\Z)]$-module defined in \citep{OS}. There is a decomposition of this group by $\SpinC$ structures,
\[\underline{\widehat{HF}}(M)=\bigoplus_{{\substack{\spinc\in\SpinC(M)}}}\underline{\widehat{HF}}(M,\spinc).
\]
\begin{theorem}\label{thm:OS}\citep[Theorem~1.1]{OS4}
	For a closed $3$-manifold $M$ and $\xi\in H_2(M,\Z)$,
	\[\Theta(\xi)=\max_{\substack{\left\{\spinc\in\SpinC(M)\ |\ \underline{\widehat{HF}}(M,\spinc)\neq0\right\}}}|\langle c_1(\spinc),\xi\rangle|.
	\]
\end{theorem}

Let us consider the case where $M=N_f$ is given as above. Extend $[\partial^+ N]$ to a basis for $H_2(N_f,\Z)\cong\Z^3$ and consider a corresponding identification of $\SpinC(M)$ with $\Z^3$ (by evaluation of the first Chern class of the $\SpinC$ structures over the generators of the homology group $H_2(M,\Z)\cong\Z^3$). In order to prove Theorem~\ref{thm3}, we show that there are two linearly independent $\SpinC$ structures $\spinc_1$ and $\spinc_{2}$, with the property that 
\[\langle c_1(\spinc_i),[\partial^+ N]\rangle=0\quad\quad\text{and}\quad\quad  \underline{\widehat{HF}}(M,\spinc_i)\neq 0\quad\quad\text{for}\ \ i=1,2.\]
 Since $\Theta([\bar{C}])=0$, we thus have $[\bar{C}]=\lambda[\partial^+ N]$, for some integer $\lambda$, which contradicts our assumption. This shows that $N$ does not have a simplifier and is thus not equivalent to $S^2\times I$.\\
 
A remark about the above argument may be appropriate here. Let us assume that $N_1$ and $N_2$ are balanced $3$-manifolds and that $N_1$ simplifies to $N_2$, i.e. $N_1\xrightarrow{C}N_{2}$. For $k=1,2$, let $M_k$ denote the closed $3$-manifold obtained by taking two copies $N_k^1$ and $N_k^2$ of $N_k$, identifying $\partial^+N_k^1$ with $\partial^+N_k^2$ and identifying $\partial^-N_k^1$ with $\partial^-N_k^2$. Then the cylinder $C$ gives the torus $T\subset M_1$, while $M_{2}$ is obtained by cutting $M_1$ along $T$ and gluing two solid tori to the resulting boundary components. Theorem~\ref{thm:OS} is then helpful in detecting $T$. Nevertheless, the equivalence of $N_1$ and $N_2$ is yet not well-translated to Heegaard Floer theory, e.g. to a practical correspondence between $\underline{\widehat{HF}}(M_1)$ and $\underline{\widehat{HF}}(M_{2})$. If $T$ is $2$-sided, the problem is studied in \cite{E-splicing} and \cite{E-essential}, and a relatively powerful machinery is developed in \cite{HRW}. For non-separating $T$,  it is interesting to develop such a correspondence.\\

{\bf{About the proof}}. In Section~\ref{sec-HD}, we construct a Heegaard diagram $\mathcal{H}$ for the closed manifold $M$ from $\overline{\Hcal}$, following the approach of \cite{lekili}. The number of generators for the Heegaard diagram $\mathcal{H}$ is $7936$, and it is thus not feasible to find the $\SpinC$ structures $\spinc_1$ and $\spinc_2$ and compute the groups $\underline{\widehat{HF}}(M,\spinc_i)$ without computational assistance (from computers). We prove a simple lemma from linear algebra in Section~\ref{Computation}, in the sprit of the general discussion in \cite[Section 2]{E-splicing}. The lemma is used, in combination with a computer program, to obtain a short-list of potential $\SpinC$ structures $\spinc$ with $\underline{\widehat{HF}}(M,\spinc)\neq 0$ (although obtaining the short-list is not an official part of our argument). Among the potential candidates, two specific $\SpinC$ structures $\spinc_1$ and $\spinc_2$ are considered in Section~\ref{Computation-1} and Section~\ref{Computation-2}. The chain complexes associated with these $\SpinC$ structures are $8$-dimensional and $72$-dimensional, respectively. The homology groups of the chain complexes $\underline{\widehat{CF}}(M,\spinc_i)$ (for $i=1,2$) are studied using the lemma proved in Section~\ref{Computation}, a series of computer assisted computations and explicit computations of the contribution of moduli spaces associated with certain classes of Whitney disks. Since the Heegaard diagram is not nice (in the sense of \cite{Sarkar-Wang}), such explicit computations are necessary and appear in Section~\ref{sec:disks}. 
\newpage

\section{A Heegaard diagram for the mapping torus}\label{sec-HD}
In this section, we obtain a Heegaard diagram for $M=N_f$, using the construction of \citep{lekili}. Let us assume that the diagram $\overline{\Hcal}$ is obtained from a Morse function $h:N\rightarrow[-1,1]$. Then $h$ gives a circle-valued Morse function $\bar{h}:M\rightarrow S^1$ with two critical points $x_1$ and $x_2$ of index $1$ and two critical points $y_1$ and $y_2$ of index $2$, such that
\begin{align*}
N_{x_1}^u(h)\cap\Sigma=\alpha_1,\quad\quad N_{x_2}^u(h)\cap\Sigma=\alpha_2,\quad\quad &N_{y_1}^s(h)\cap\Sigma=\beta_1,\quad\quad \ N_{y_2}^s(h)\cap\Sigma=\beta_2,\\
\bar{h}^{-1}(1)=\partial^+N\sim_f\partial^-N=\Sigma_{min}\quad\quad\quad\quad\quad\quad&\text{and}\quad\quad\quad\quad\quad\quad\bar{h}^{-1}(-1)=\bar\Sigma=\Sigma_{max}.
\end{align*}
Here $N_x^s$ and $N_x^u$ denote the stable and unstable manifold of $x\in N$ with respect to the flow of a gradient-like vector field for $h$. Following \citep{lekili}, let $p_1$ and $p_2$ be  disjoint points in $\bar\Sigma\setminus\alphas\cup\betas$ and $\gamma_1$ and $\gamma_2$ denote two gradient flow lines disjoint from $N^u_{x_i}$ and $N^s_{y_i}$ such that 
\[\gamma_i\cap\bar\Sigma=\{p_i\}\quad\text{and} \quad\gamma_i\cap\partial^+N=\{\bar{p}_i\},\quad\quad\text{for} \ \ i=1,2.\] 
Furthermore, $\gamma_1$ (resp. $\gamma_2$) is mapped onto the northern (resp. southern) semi-circle of $S^1$. Let $N(\gamma_i)$, $i=1,2$, denote the normal neighborhood of $\gamma_i$ that intersects $\bar\Sigma$ and $\partial^+N$ in the small disks $D_{p_i}$ and $D_{\bar{p}_i}$, respectively. By removing $D_{p_i}^\circ$ and $D_{\bar{p}_i}^\circ$ and gluing $\partial D_{p_i}$ to $\partial D_{\bar{p}_i}$ along $\partial N(\gamma_{i})$ we obtain the Heegaard surface $\Sigma$. Let $\alpha_{5}=\partial D_{\bar{p}_1}$ and $\beta_5=\partial D_{\bar{p}_2}$. Let $\alpha_3'$ and $\alpha_{4}'$ (resp. $\beta_3'$ and $\beta_4'$),  be disjoint arcs in $\partial^+N$ such that $\partial \alpha_3'$ and $\partial \alpha_{4}'$ are disjoint points on $\beta_5$ and $\partial \beta_3'$ and $\partial \beta_4'$ are disjoint points on $\alpha_{5}$, while $|\alpha_3'\cap \beta_4'|=|\alpha_4'\cap \beta_3'|=1$ and $|\alpha_3'\cap \beta_3'|=|\alpha_{4}'\cap \beta_4'|=0$.  Flowing the arcs $\beta_3'$ and $\beta_4'$ through the gradient flow of $\bar{h}$ above the northern semi-circle, we obtain disjoint arcs $\beta_3''$ and $\beta_4''$ in $\Sigma\setminus\partial^+N$ which are disjoint from $\beta_1$ and  $\beta_2$. Similarly, flowing the arcs $\alpha_3'$ and $\alpha_{4}'$, we obtain  $\alpha_3''$ and $\alpha_{4}''$ which are disjoint from $\alpha_1$ and $\alpha_2$. This determines the sets of $\alpha$ and $\beta$ curves:
\[\alphas=\{\alpha_1,\alpha_2,\alpha_3=\alpha_3'\cup \alpha_3'',\alpha_{4}=\alpha_{4}'\cup \alpha_{4}'',\alpha_{5}\}\ \ \text{and}\ \ \betas=\{\beta_1,\beta_2,\beta_3=\beta_3'\cup \beta_3'',\beta_4=\beta_4'\cup \beta_4'',\beta_5\}.\]

\begin{figure}
	\begin{center}
		\def\svgwidth{0.8\textwidth}
		\fontsize{8}{10}\selectfont
		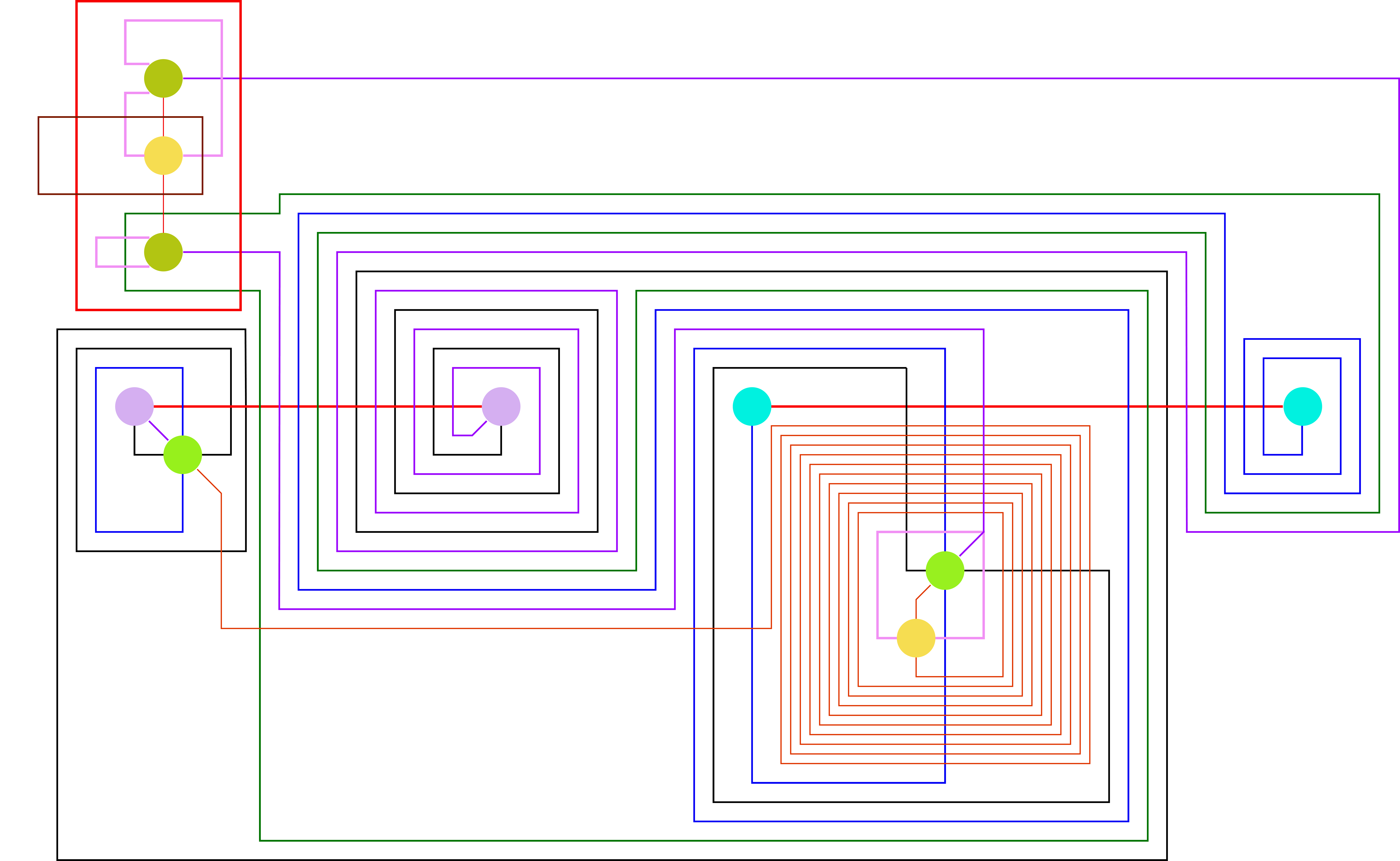	
		\caption{A weakly admissible Heegaard diagram of $N_f$ with $21392$ generators}
		\label{ex-3}
	\end{center}
\end{figure}

\begin{figure}
	\begin{center}
		\def\svgwidth{\textwidth}
		\fontsize{7}{10}\selectfont
		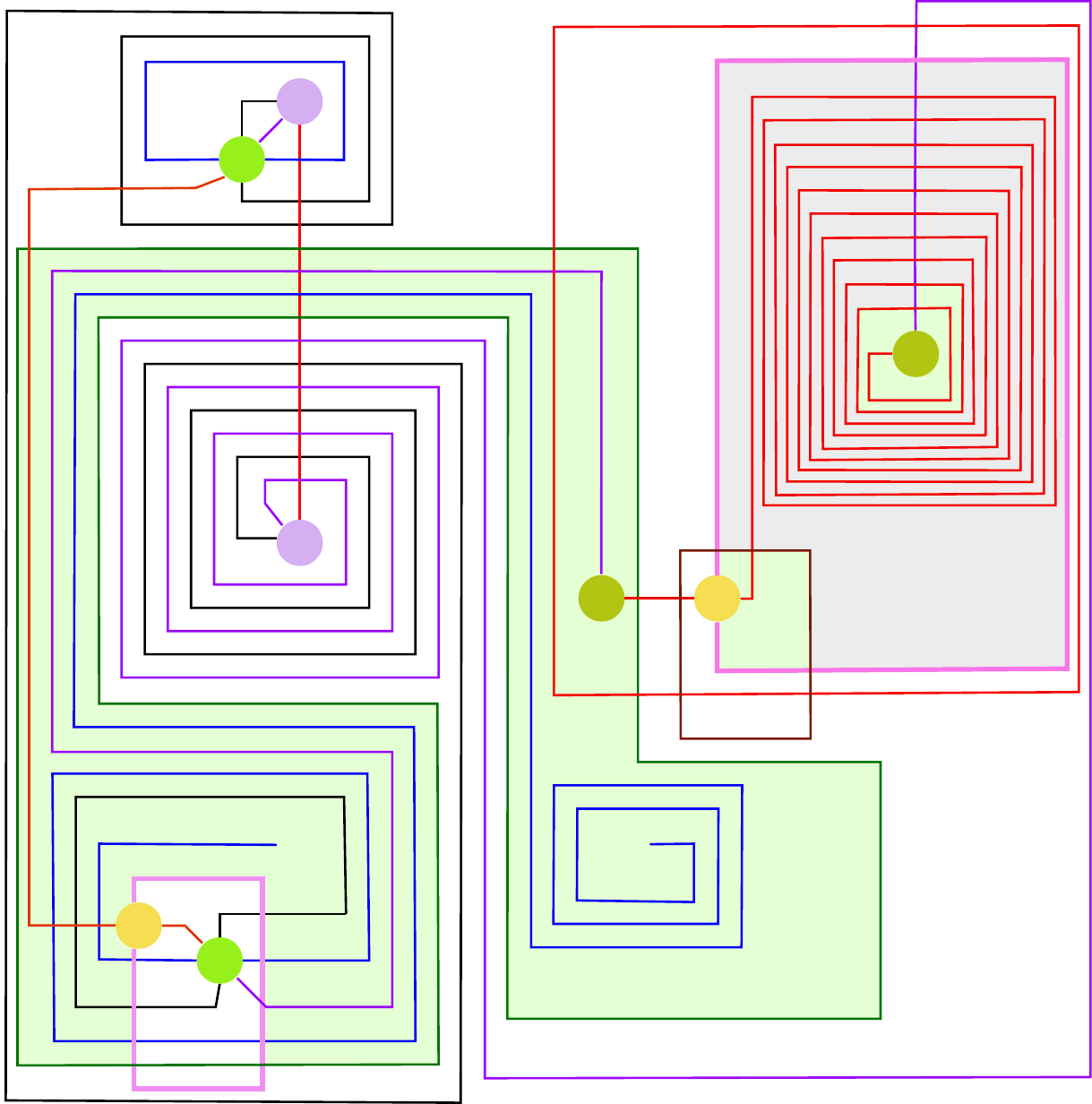	
		\caption{A weakly admissible Heegaard diagram for $M$ with $7936$ generators. The connected components of $\Sigma\setminus{\alphas\cup\betas}$ are labeled  $D_i$, for $i=0,\ldots,67$.  The periodic domains are generated by $P_1=-D_{52}+\sum_{i=53}^{57}D_{i}+\sum_{i\in I_1} D_i$, $P_2=-D_{49}+\sum_{i\in I_2}D_{i}+\sum_{i\in I_1} D_i$ and a third periodic domain $P_3$, where $D_i$ is colored gray for $i\in I_1$ and green for $i\in I_2$. We have $\partial_b(P_1)=-\beta_5$ and $\partial_b(P_2)=\beta_3$. The periodic domain $P_3$ may be chosen so that $\partial_b(P_3)=\beta_4+2\beta_3-3\beta_1+2\beta_2$.}
			\label{ex-4}
		\end{center}
	\end{figure}

Having fixed a marked point $z$, finger-move  isotopies may be used to make $(\Sigma,\alphas,\betas,z)$ weakly admissible. If we apply the procedure to the Heegaard diagram of Figure~\ref{ex-1}, we arrive at the admissible Heegaard diagram illustrated in Figure~\ref{ex-3} with $21392$ generators. Handle-slides of $\alpha_{4}$ over $\alpha_3$ ($10$ times) and isotopies on  $\alpha_3$ give an alternative  (more suitable) weakly admissible Heegaard diagram $\mathcal{H}$ with $7936$ generators,  as illustrated in Figure \ref{ex-4}.  We use  $x_{i,j,k}$ to label the intersection point of $\alpha_i$ and $\beta_j$ which is labeled $k$ in the diagram of Figure~\ref{ex-4}.\\

	The set of periodic domains for $\Hcal$ is generated by three domains $P_1$, $P_2$ and $P_3$. The first two generators are shown in Figure \ref{ex-4}. The periodic domains $P_1$ and $P_2$ are of the form
\begin{align*}
P_1=-D_{52}+\sum_{i=53}^{57}D_{i}+\sum_{i\in I_1} D_i\quad\text{and}\quad P_2=-D_{49}+\sum_{i\in I_2}D_{i}+\sum_{i\in I_1} D_i,
\end{align*}	
where the domains $D_i$ with $i$ in $I_1$ and $I_2$ are colored gray and green in Figure~\ref{ex-4}, respectively. If $\partial_bP$ denote the $\beta$-boundary of a periodic domain $P$, we then have
$\partial_b(P_1)=-\beta_5$ and $\partial_b(P_2)=\beta_3$. We may choose the third generator $P_3$ of the space of periodic domains so that $\partial_b(P_3)=\beta_4+2\beta_3-3\beta_1+2\beta_2$. Let $H(P_i)\in H_2(M,\Z)$ denote the homology classes associated with the periodic domains $P_i$ for $i=1,2,3$, which form a basis for $H_2(M,\Z)$ (see \citep{OS}, Proposition 2.15). Correspondingly, we obtain a bijection
\begin{align*}
c:\SpinC(M)\ra \Z\oplus\Z\oplus\Z,\quad\quad c(\spinc):=\frac{1}{2}\big(\langle c_1(\spinc),H(P_1)\rangle,\langle c_1(\spinc),H(P_2)\rangle,\langle c_1(\spinc),H(P_3)\rangle\big),
\end{align*}	
which gives an identification of $\SpinC(M)$ with $\Z^3$. To compute  $\spinc_z:\mathbb{T}_{\alphas}\cap\mathbb{T}_{\betas}\rightarrow\SpinC(M)=\Z^3$ under this identification, define $\spinc^r_z:\mathbb{T}_{\alphas}\cap\mathbb{T}_{\betas}\rightarrow \Z^3$ by setting $\spinc^r_z(x_0)=(0,0,0)$ for
\begin{align*}
x_0=(x_i^0)_{i=1}^5=(x_{1,1,1},x_{2,2,1},x_{3,4,1},x_{4,5,1},x_{5,3,2}).
\end{align*}
Let $(y_i^0)_{i=1}^5$ denote a permutation of $(x_i^0)_{i=1}^5$ so that $y_i^0\in\beta_i$. Fix a connected path $\gamma_0$ on $\alphas\cup\betas$ in the diagram such that for each $\alpha\in\alphas$ and $\beta\in\betas$,  $\gamma_0\cap\alpha$ and $\gamma_0\cap\beta$ are connected and $x_i^0\in\gamma_0$, $1\leq i\leq5$, (the yellow path in Figure \ref{ex-4} satisfies these properties). Fix 
\[x=(x_i)_{i=1}^5=(y_i)_{i=1}^5\in \mathbb{T}_{\alphas}\cap\mathbb{T}_{\betas}\quad\quad \text{with}\ \  x_i\in\alpha_i\ \ \text{and}\ \  y_i\in\beta_i,\] 
i.e. $(y_i)_i$ is just a permutation of $(x_i)_i$.  Let $\epsilon(x_0,x)$ denote the closed $1$-cycle in $\Sigma$ obtained by connecting $y_i$ to $y_i^0$ through $\beta_i$,  connecting $y_i^0$ to $x_i^0$ through $\gamma_0$, and connecting $x_i^0$ to $x_i$ through $\alpha_i$ for $i=1,\ldots,5$. Note that for $j=1,2,3$, the evaluation $\langle PD[\epsilon(x_0,x)],H(P_j)\rangle$ is the algebraic intersection number of $\epsilon(x_0,x)$ with $\partial_{b} P_j$. Therefore, if we set 
\begin{align*}
\spinc^r_z(x):=\big(-\langle \epsilon(x_0,x),\beta_5\rangle,\langle \epsilon(x_0,x),\beta_3\rangle,\langle\epsilon(x_0,x),\beta_4+2\beta_3-3\beta_1+2\beta_2\rangle\big),
\end{align*}
there is a fixed triple $(a,b,c)=(0,-1,-4)
\in\Z^3$ such that $\spinc_z=\spinc^r_z(x)+(a,b,c)$. In the definition of $\spinc^r_z$, note that the  intersection numbers take place over the Heegaard surface. 
The map $\spinc_z^r$ is used instead of $\spinc_z$ for the purposes of this paper.
\section{Simplifying computations using  algorithmic calculations} \label{Computation}
All our computations are performed with coefficients in $\Z_2[H^1(M,\Z)]$. In the discussions of this section, we have the  diagram $\mathcal{H}=(\Sigma,\alphas,\betas,z)$ from Figure~\ref{ex-4} in mind. Nevertheless, the strategy works for many of the chain complexes associated with sutured manifold diagrams in the sense of \cite{Juhasz-sutured}, or even \cite{AE-sutured}.  Since there is a large number of generators associated with $\mathcal{H}$,  we break the computation of $\underline{\widehat{HF}}(M)$ into a computer-assisted part and a human part using the following observation. \\

Let  $\mathbf{z_2}\subset\mathbf{z_1}$ denote two sets of marked points containing $z$. Most of the time, we take $\mathbf{z_2}=\{z\}$. If $\mathbf{z_1}$ is sufficiently large so that it contains a marked point in each one of the periodic domains, we may choose a decomposition $\underline{\widehat{CF}}(\Sigma,\alphas,\betas,\mathbf{z_1})=A\oplus B\oplus H$, so that the differentials  $d_{\mathbf{z_1}}$ and $d_{\mathbf{z_2}}=d_{\mathbf{z_1}}+d'$ are determined by the matrices
\begin{equation}\label{eq-2}
d_{\mathbf{z_1}}=\left(\begin{array}{ccc}
	0 & 0 & 0\\ I & 0 & 0 \\ 0 & 0 & 0
\end{array}
\right)\quad\text{and}\quad
d'=\left(\begin{array}{ccc}
	f&h&m\\
	k& g&n\\
	p&q&l
\end{array}
\right).
\end{equation}

\begin{lemma}\label{HF}
	Suppose that with the above notation in place,  $I+k$ is invertible. Then 
 \[{H_*}(\underline{\widehat{CF}}(\Sigma,\alphas,\betas,\mathbf{z_2}))={H_*}(H,l+p(I+k)^{-1}n).\]
\end{lemma}
\begin{proof}
	The proof follows from two base changes. The first base change is given by
	\[
		\begin{split}
		\left(\begin{array}{ccc}
			I+k&(I+k)f(I+k)^{-1}&0\\0&I&0\\0&0&I
		\end{array}
		\right)
		\left(\begin{array}{ccc}
			f&h&m\\I+k&g&n\\p&q&l
		\end{array}
		\right)
		\left(\begin{array}{ccc}
			(I+k)^{-1}&f(I+k)^{-1}&0\\0&I&0\\0&0&I
		\end{array}
		\right)\\
		=\left(\begin{array}{ccc}
			0&*&m'\\I&*&n\\p(I+k)^{-1}&*&l
		\end{array}
		\right)
		=\left(\begin{array}{ccc}
			0&m'A&m'\\I&nA&n\\A&lA&l
		\end{array}
		\right),
		\end{split}
	\]
	where $A=p(I+k)^{-1}$ and the last equality follows from  $d_{\mathbf{z_2}}^2=0$. The second base change is
	\	\[\begin{split}
		\left(\begin{array}{ccc}
			I&0&0\\0&I&0\\0&A&I
		\end{array}
		\right)
		\left(\begin{array}{ccc}
			0&m'A&m'\\I&nA&n\\A&lA&l
		\end{array}
		\right)
		\left(\begin{array}{ccc}
			I&0&0\\0&I&0\\0&A&I
		\end{array}
		\right)=\left(\begin{array}{ccc}
			0&0&m'\\I&0&n\\0&0&l+An
		\end{array}
		\right).
	\end{split}
	\]
	
\end{proof}

In applications of Lemma~\ref{HF}, we choose an area assignment $\Acal$ for the  regions in $\Sigma\setminus\alphas\cup\betas$ such that $\Acal(P_i)=0$ for $i=1,2,3$. Moreover,  $\mathbf{z_1},\mathbf{z_2}$ and $\Acal$ are chosen so that the regions  not touched by  $\mathbf{z_1}$ have very small areas and the regions containing marked points from $\mathbf{z_1}\setminus\mathbf{z_2}$ have very large areas.  Under these assumptions, in each $\SpinC$ class $\spinc$, $\Acal$ descends to an energy filtration on $\underline{\widehat{CF}}(\Sigma,\alphas,\betas,\mathbf{z_2},\spinc)$ (see \cite{OS}). We may further assume that
\[A=\langle a_1,\dots,a_r\rangle,\ \  B=\langle b_1,\dots,b_r\rangle,\quad\quad\text{with}\ \ 
\mathcal{A}(a_1)<\mathcal{A}(a_2)<\cdots<\mathcal{A}(a_r),\] 
while the differential $d_{\mathbf{z_1}}$ is given by sending $a_i$ to $b_i$. With respect to the energy filtration, $\mathcal{A}(a_i)-\mathcal{A}(b_i)$ is then a small positive number and $k$ is a  lower triangular matrix with zeros on the diagonal. Therefore, $I+k$ is an invertible matrix with $(I+k)^{-1}=\sum_{i=0}^\infty k^i$. This allows us use Lemma~\ref{HF}. Of course, the use of Lemma~\ref{HF} is not restricted to the aforementioned situation.\\

In our search for the $\SpinC$ classes $\spinc$ with the property that $\underline{\widehat{HF}}(M,\spinc)\neq 0$, we may first restrict our attention to the $\SpinC$ classes which satisfy $\langle c_1(\spinc), H(P_1)\rangle=0$, since $P_1$ corresponds to $\partial^+N$ and is represented by an embedded surface of genus $1$. We may then enlarge the set $\mathbf{z_2}=\{z\}$ of punctures in the Heegaard diagram to a bigger set $\mathbf{z_1}$, so that $(\Sigma,\alphas,\betas,\mathbf{z_1})$ is nice, while the criteria discussed in the previous two paragraphs is satisfied. If the group \[H_*(\underline{\widehat{HF}}(\Sigma,\alphas,\betas,\mathbf{z_1},\spinc)\]
is trivial, it  follows that $\underline{\widehat{HF}}(M,\spinc)$ is also trivial. Trying different sets $\mathbf{z_1}$ of marked points allows us exclude many of the $\SpinC$ classes $\spinc$ from the the Thurston polytope, consisting of $\SpinC$ structures $\spinct$ with  $\underline{\widehat{HF}}(M,\spinct)\neq 0$. Among the remaining $\SpinC$ classes, we combine Lemma~\ref{HF}, computer assisted computations and the study of certain classes of Whitney disks (from next section) to show that $\underline{\widehat{HF}}(M,\spinc_i)\neq 0$ for $i=1,2$, where $\spinc_1$ and $\spinc_2$ are the classes of generators $x_1$ and $x_2$ of $\underline{\widehat{CF}}(\Hcal)$ with
\[\spinc_z^r(x_1)=(0,1,7)\quad\text{and}\quad \spinc_z^r(x_2)=(0,-1,-8),\]
respectively. As we will see in Section~\ref{Computation-1} and Section~\ref{Computation-2}, there are $8$ generators $x_1$ and $72$ generators $x_2$ of the above type.

\newpage
\section{Non-polygonal disks with holomorphic representatives}\label{sec:disks}
In this section, we study the moduli spaces associated with three classes of Whitney disks with non-polygonal domains, which will be encountered in Section~\ref{Computation-2}. First, let $D_{k,l,n}=D(\phi_{k,l,n})$ denote the genus zero domain of a Whitney disk $\phi_{k,l,n}$, with two boundary components having $2k$-edges and $2l$-edges, respectively. The edges on each boundary component consist of alternating arcs from distinct $\alpha$ and $\beta$ curves. For such a disk to have Maslov index $1$, it is necessary that all the $2(k+l)$ angles on the boundary are acute angles, except for precisely one of them. We further assume that the obtuse angle is on the boundary component with $2l$ edges, where $\alpha_1$ and $\beta_1$ meet at $x_n$ and enter the interior of $D_{k,l,n}$, and intersect each other at $x_{n-1},\ldots,x_1$ in $D_{k,l,n}^\circ$.  There is some extra freedom in choosing the domain $D_{k,l,n}$ (up to isotopy of the curves) which corresponds to the edges where $\alpha_1$ and $\beta_1$ exit $D_{k,l,n}$ and is dropped from the notation (see Figure \ref{disk-1} (left)).\\

\begin{figure}
	\begin{center}
		\def\svgwidth{0.9\textwidth}
		\fontsize{8}{10}\selectfont
		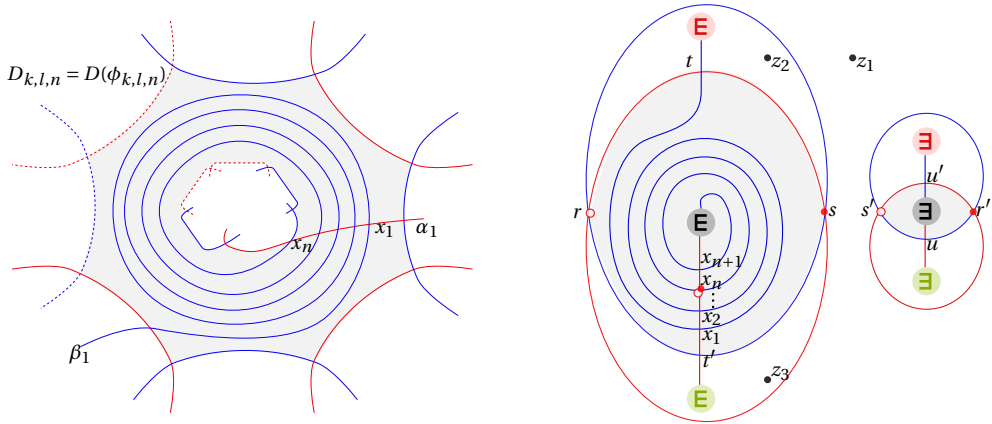	
		\caption{Part of a  Heegaard diagram which illustrates the domain associated with the disk $\phi_{k,1,n}$ is illustrated (left). The red curves are the $\alpha$ curves and the blue curves are the $\beta$ curves. A Heegaard diagram of genus $3$ containing the domain $D_{1,1,n}$ is illustrated on the right. The domain of the Whitney disk $\psi\in\pi_2(R_n,U_n)$ is shaded.} 
		\label{disk-1}
	\end{center}
\end{figure}

\begin{lemma}\label{Non-Poly-1}
	Let $\phi_{k,l,n}$ be a disk with a domain as described above. Then $\#\Mhat (\phi_{k,l,n})=1$.
\end{lemma}
\begin{proof}
	First, consider the case $k=l=1$. Consider the triply punctured Heegaard diagram
	\[\mathcal{H}_1=\mathcal{H}_1^n=(\Sigma_1,\boldsymbol{\alpha_1}=\{\alpha_1,\alpha_2,\alpha_3\},\boldsymbol{\beta_1}=\{\beta_1,\beta_2,\beta_3\},z_1,z_2,z_3)\] 
illustrated in Figure \ref{disk-1} (right). Here $\Sigma_1$ is a surface of genus three and the sutured manifold determined by $\mathcal{H}_1$ is the same as the sutured manifold determined by a Heegaard diagram $(T=S^1\times S^1,\alpha,\beta=\{b\}\times S^1,z_1,z_2)$, where $\alpha$ is homotopically trivial and cuts $\beta$ twice, one of the punctures is located in one of the two bigons in $T-\alpha-\beta$, and two of the punctures are located in the cylindrical component of $T-\alpha-\beta$. Therefore, the Heegaard Floer group associated with $\mathcal{H}_1$ is trivial. With the notation of Figure~\ref{disk-1}, the generators in  $\mathbb{T}_{\alphas}\cap\mathbb{T}_{\betas}$ are 
	\begin{align*}
	&R_i=(x_i,\ r,\ r'), &&S_i=(x_i,\ r,\ s'),&&T_i=(x_i,\ s,\ r'), & & U_i=(x_i,\ s,\ s'), && i=1,\dots,n+1\\
	&V=(t,\ t',\ r'), && W=(t,\ t',\ s'), &&X=(u,\ r,\ u'), && Y=(u,\ s,\ u'). &&
	\end{align*}
Most Whitney disks with positive domain and index $1$ which contribute to the differential are of the form $\phi_{1,1,k}$ for some $k=1,\ldots,n$. In fact, there are Whitney disks
\begin{align*}
\psi_k^1\in\pi_2(U_{k},S_{k+1}),\quad\psi_k^2\in\pi_2(T_{k},R_{k+1}),\quad \psi_{n+1-k}^3\in\pi_2(R_{k+1},S_k)\quad\text{and}\quad \psi_{n+1-k}^4\in\pi_2(T_{k+1},U_k)
\end{align*}
for $k=1,\ldots,n$, where each $\psi_k^i$ is of type $\phi_{1,1,k}$. Other than these classes, there are also disks 
\[\psi^1_0\in\pi_2(V,R_1),\quad \psi^2_0\in\pi_2(W,S_1),\quad\psi^3_0\in\pi_2(X,S_{n+1})\quad\text{and}\quad \psi^4_0\in\pi_2(Y,U_{n+1}),\]		
with Maslov index one, and the domain of every one of them is a rectangle. Therefore, $\#\Mhat(\psi_0^i)=1$ for $i=1,\ldots,4$. 
Moreover, there are disks $\phi\in\pi_2(T_1,W)$ and $\phi'\in\pi_2(T_{n+1},X)$ with domains of type $\phi_{1,2,n}$. If we set 
\[m=\#\Mhat(\phi),\quad m'=\#\Mhat(\phi')\quad \text{and}\quad
m_k^i=\#\Mhat(\psi_k^i)\quad \text{for}\ \  i=1,\ldots,4\ \  \text{and}\ \  k=0,\ldots,n,\] 
it follows that $m_0^i=m_1^i=1$ and that the differential of the chain complex is given by
\begin{diagram}
T_k&\rTo{m_k^2}&R_{k+1}&&&T_1&\rTo{1}&R_2&&T_{n+1}&\rTo{m'}&X&&Y&&V\\
\dTo{m_{n+2-k}^4}&&\dTo_{m_{n+1-k}^3}&&&\dTo{m}&&\dTo_{m_n^3}&&\dTo{1}&&\dTo_{1}&&\dTo{1}&&\dTo{1}\\
U_{k-1}&\rTo_{m_{k-1}^1}&S_k&&&W&\rTo_{1}&S_1&&U_n&\rTo_{m_n^1}&S_{n+1}&&U_{n+1}&&R_1
\end{diagram}
for $k=2,\ldots, n$. Therefore, we conclude that $m=m_n^3$, $m'=m_n^1$ and 
\begin{equation}\label{eq:1}
m_k^2\cdot m_{n+1-k}^3=m_{n+2-k}^4\cdot m_{k-1}^1\quad\quad\text{for}\ \ k=2,\ldots,n.
\end{equation}
For $n=2$, (\ref{eq:1}) implies $m_2^2=m_2^4$. Moreover, since the homology is trivial,  $m_2^2=m_2^4=1$. This proves the claim for $\phi_{1,1,2}$. Having established the proof for $\phi_{1,1,j}$ with $j=1,\ldots,n-1$ (where $n>2$), Equation~(\ref{eq:1}) for $k=2$ implies that $m_n^4=1$, proving the claim for $\phi_{1,1,n}$.\\
	
	\begin{figure}
		\begin{center}
			\def\svgwidth{0.78\textwidth}
			\fontsize{8}{10}\selectfont
			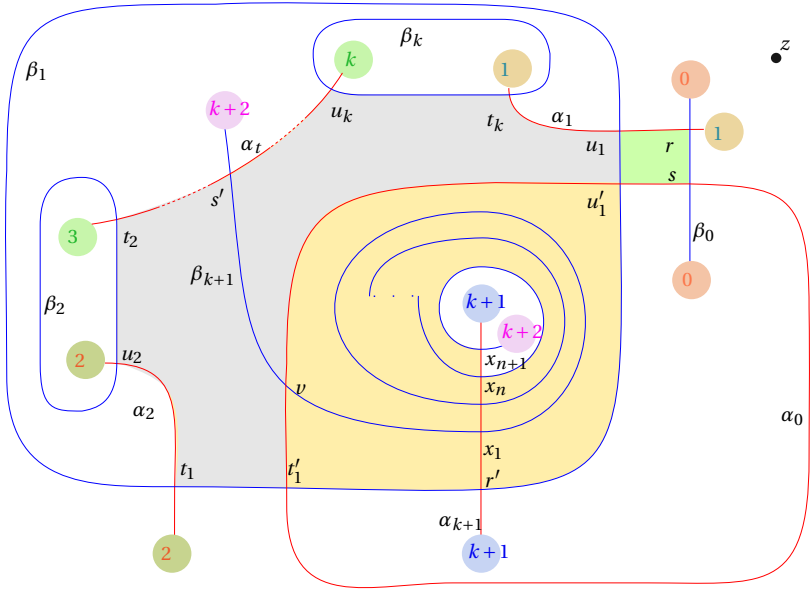	
			\caption{A  Heegaard surface of genus $k+3$.The colored region denotes the domain associated with  $\psi$, which degenerates in two ways as $\phi_{1,1,n}*\phi$ and $\phi'*\phi_{k,1,n}$, where $D(\phi')$ is the region colored green and  $D(\phi_{1,1,n})$ is the union of regions colored yellow.}
			\label{disk-3}
		\end{center}
	\end{figure} 

Next, we consider the case  $l=1$ while $k$ is arbitrary. Let
	\[\mathcal{H}_2=(\Sigma_2,\boldsymbol{\alpha_2}=\{\alpha_0,\alpha_1,\dots,\alpha_{k+1}\},\boldsymbol{\beta_2}=\{\beta_0,\beta_1,\dots,\beta_{k+1}\},z)\] 
	be the Heegaard diagram shown in Figure \ref{disk-3}. Here $\Sigma_2$ is a surface of genus $k+3$.  With the notation of Figure~\ref{disk-3} in place and refreshing the notation set for the case $k=l=1$, the  generators in $\mathbb{T}_{\alphas}\cap\mathbb{T}_{\betas}$ are 
\begin{align*}
&W=(v,r,u_2,\dots,u_k,r'),
&&V=(s,t_k,u_2,\dots,u_{t-1},s',t_t,\dots,t_{k-1},r'),
&&R_i=(u'_1,r,u_2,\dots,u_k,x_i),\\ 
&S_i=(t'_1,r,u_2,\dots,u_k,x_i),
&&T_i=(s,t_k,t_1,\dots,t_{k-1},x_i)\quad\quad\text{and}
&&U_i=(s,u_1,u_2,\dots,u_k,x_i),
\end{align*}
where $i$ belongs to $\{1,\dots,n+1\}$. Consider the Whitney disks
\begin{align*}
&\phi_{1,1,n}\in\pi_2(R_n,S_{n+1}), 
&&\phi\in\pi_2(S_{n+1},T_{n+1}),
&&\phi'\in\pi_2(R_n,U_n), 
&&\phi_{k,1,n}\in\pi_2(U_n,T_{n+1}).
\end{align*}
such that $D(\phi')$ is the green domain, $D(\phi_{1,1,n})$ is the union of yellow domains, $D(\phi)$ is the union of grey and green domains and $D(\phi_{k,1,n})$ is the union of grey and yellow domains. Then $\psi=\phi_{1,1,n}*\phi=\phi'*\phi_{k,1,n}$ has index $2$, while these are the only degenerations of $\psi$ as a juxtaposition of two positive Whitney disks of Maslov index $1$.  This implies 
\begin{align*}
\#\Mhat (\phi_{k,1,n})=\#\Mhat (\phi_{k,1,n})\cdot \#\Mhat (\phi')=
\#\Mhat (\phi_{1,1,n})\cdot \#\Mhat (\phi)=\#\Mhat (\phi_{1,1,n})=1,
\end{align*}
completing the proof for the case where $l=1$, while $k$ and $n$ are arbitrary. Similarly, the argument above may be used to conclude  $\#\Mhat (\phi_{k,l,n})=1$ for arbitrary values of $k,l$ and $n$. 		
\end{proof}

\begin{figure}
\begin{center}
	\def\svgwidth{0.9\textwidth}
	\fontsize{8}{10}\selectfont
	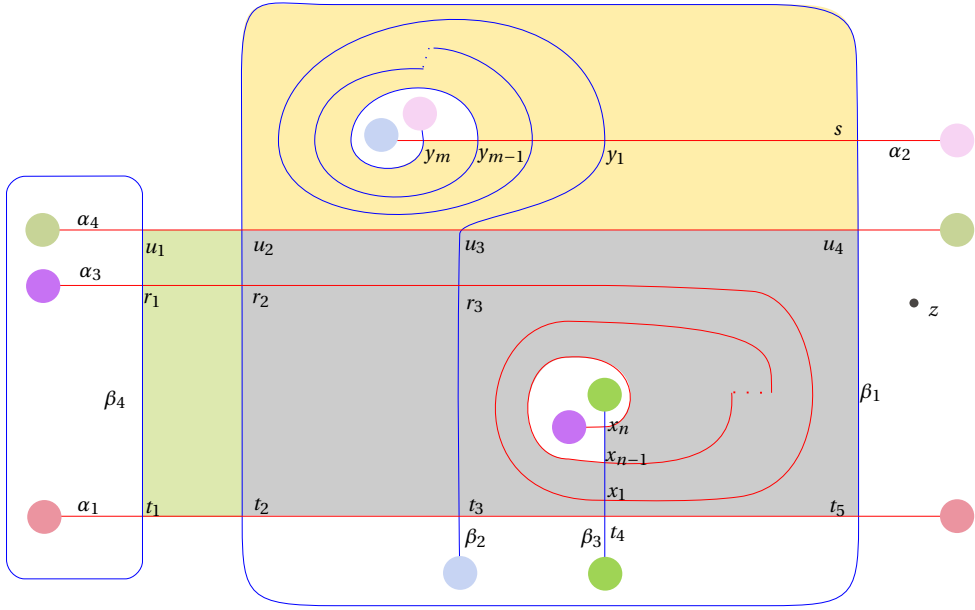	
	\caption{A  Heegaard surface of genus $6$.  The colored regions denote the domain associated with a Whitney disk $\psi$ of index $2$ which degenerates in two ways.}
	\label{disk-4}
\end{center}
\end{figure}

Let $D(\phi_{n,m})$ denote the genus zero domain of a Whitney disk $\phi_{n,m}$ which has three boundary components, each consisting of $2$ edges  on  $\alpha_i$ and $\beta_i$ for $i=1,2,3$. Let $\alpha_3$ have $n$ intersection points $\{x_1,\dots,x_n\}$ with $\beta_3$  and $\alpha_2$ have $m$ intersection points $\{y_1,\dots,y_m\}$ with $\beta_2$ in $D(\phi_{n,m})$. The union of the yellow regions and the grey regions in Figure~\ref{disk-4} illustrates the domain of such a disk. We assume that all the corners of the boundary edges in $D(\phi_{n,m})$ are acute except for two, where $\alpha_2$ intersects $\beta_2$ in an obtuse angle in $y_{m-1}$ and $\alpha_3$ intersects $\beta_3$ in an obtuse angle in $x_{n-1}$ (see Figure~\ref{disk-4}).

\begin{lemma}\label{Non-Poly-2}
	If the domain of $\phi_{n,m}$ is as described above, we have $\#\Mhat (\phi_{n,m})=1$.
\end{lemma}
\begin{proof}
Consider the Heegaard diagram
	\[\mathcal{H}_3=(\Sigma,\alphas=\{\alpha_1,\alpha_2,\alpha_3,\alpha_{4}\},\betas=\{\beta_1,\beta_2,\beta_3,\beta_4\},z)\] 
	which is illustrated in Figure \ref{disk-4}. Here $\Sigma$ is a surface of genus six which is obtained by attaching six one-handles that each one connects the boundary circles of disks with the same color.  There are $4nm+3m+2n+2$ intersection points in $\mathbb{T}_{\alphas}\cap\mathbb{T}_{\betas}$. With the notation of Figure~\ref{disk-4} in place, these intersection points are
\begin{align*}
&P=(t_4,u_1,r_3,s),
&&Q=(t_4,u_3,r_1,s)
&&R_i=(t_1,u_3,x_i,s), 
&&S_i=(t_3,u_1,x_i,s),\\ 
&T_{i,j}=(t_2,u_1,x_i,y_j),
&& U_{i,j}=(t_5,u_1,x_i,y_j),
&& V_{i,j}=(t_1,u_2,x_i,y_j),
&&W_{i,j}=(t_1,u_4,x_i,y_j)\\	
&X_j=(t_4,u_2,r_1,y_j),
&&Y_j=(t_4,u_4,r_1,y_j)
&&\quad\quad\text{and}
&&Z_j=(t_4,u_1,r_2,y_j),
\end{align*}
for $i=1,\ldots,n$ and $j=1,\ldots,m$. Consider the Whitney disks of index $1$ 
	\begin{align*}
&\phi_{1,1,m-1}\in\pi_2(V_{n-1,m-1},W_{n-1,m}), &&&& \phi_{2,1,n-1}\in\pi_2(W_{n-1,m},U_{n,m}),\\
&\phi\in\pi_2(V_{n-1,m-1},T_{n-1,m-1}) &&\text{and}&& \phi_{n,m}\in\pi_2(T_{n-1,m-1},U_{n,m})
	\end{align*}
Here, the domains $D(\phi_{1,1,m-1})$ is the union of the regions colored yellow, $D(\phi)$ is the union of the regions colored green,$D(\phi_{2,1,n-1})$ is the union of the regions colored grey and green,  and $D(\phi_{n,m})$ is the union of the regions colored yellow and grey in Figure~\ref{disk-4}. Then 
\[\psi=\phi_{1,1,m-1}*\phi_{2,1,n-1}=\phi*\phi_{n,m},\]
determined by $D(\psi)$ which is the union of all colored regions, is a Whitney disk of Maslov index $2$ in $\pi_2(V_{n-1,m-1},U_{n,m})$. The disk $\psi$ degenerates as the juxtaposition of two disks of Maslov index $1$ only in the above two ways. Therefore, we conclude that
\begin{align*}
\#\Mhat(\phi_{n,m})=\#\Mhat(\phi_{n,m})\cdot \#\Mhat(\phi)=
\#\Mhat(\phi_{1,1,m-1})\cdot \#\Mhat(\phi_{2,1,n-1})=1.
\end{align*}
The last equality, which follows from Lemma~\ref{Non-Poly-1}, completes the proof of the lemma.
\end{proof}

For  $\phi\in\pi_2(x,y)$, let $D(\phi)$ be a surface of genus one with one boundary component consisting of $2$ edges that contains a unique intersection point $u$ in the interior which belongs to both $x$ and $y$ (see Figure \ref{72-3}). The grey domains on the left illustrate $D(\phi)$. 

\begin{lemma}\label{Non-Poly-3}
	Let $\phi$ be a disk with a domain as described above. Then $\#\Mhat (\phi)=1$.
\end{lemma}

\begin{figure}
	\begin{center}
		\def\svgwidth{13cm}
		\fontsize{10}{10}\selectfont
		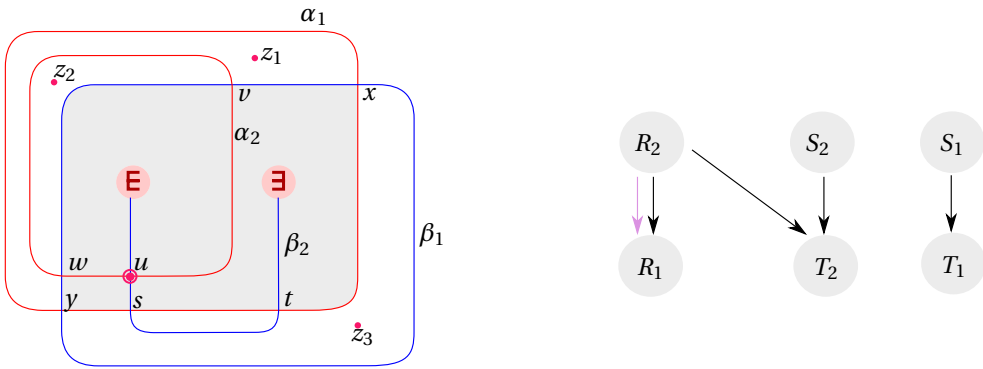	
		\caption{Left: A Heegaard diagram with three marked points and a Heegaard surface of genus one. Right: the differential associated with this diagram. }
		\label{72-3}
	\end{center}
\end{figure}

\begin{proof}
	Consider the triply punctured Heegaard diagram
	\[\mathcal{H}_4=(\Sigma,\alphas=\{\alpha_1,\alpha_2\},\betas=\{\beta_1,\beta_2\},\mathbf{z}=\{z_1,z_2,z_3\})\] 
of genus $1$ which is illustrated in Figure \ref{72-3} (left). With the notation of Figure~\ref{72-3} in place, there are $6$ intersection points in $\mathbb{T}_{\alphas}\cap\mathbb{T}_{\betas}$, which may be listed as
\begin{align*}
&R_1=(x,u),&&R_2=(y,u),&&S_1=(t,v),&&S_2=(t,w),&&T_1=(s,v)&&\text{and}&&T_2=(s,w).
\end{align*}
	The differential is shown in Figure \ref{72-3}, on the right. A black arrow which connects a generator $X$ to a generator $Y$ denotes that there is a disk from $X$ to $Y$ with a unique holomorphic representative. The arrow in purple denotes the disk with the domain $D(\phi)$. By doing an isotopy which removes the two intersection points of $\beta_2$ with $\alpha_1$ and then doing a destabilization which removes $\alpha_2$ and $\beta_2$, we obtain the standard genus zero Heegaard diagram for the closed three manifold $S^1\times S^2$.  Therefore the Heegaard Floer homology group associated with $\mathcal{H}$ is $\Z^2$. This proves that $\#\Mhat (\phi)=1$.
\end{proof}

\newpage
\section{The first non-trivial Heegaard-Floer group}\label{Computation-1}
 Let us assume that $\spinc_1$ corresponds to the triple $(0,1,7)$. The chain complex $\underline{\widehat{CF}}(M,\spinc_1)$ is then generated by the following $8$ generators:  
\begin{align*}
&\circled{1}=\big
\{x_{1,1,2}, x_{2,2,5}, x_{3,3,1}, x_{{4},5,2}, x_{{5},4,2}\big\}         
&&\circled{2}=\big
\{x_{1,1,2}, x_{2,2,5}, x_{3,3,2}, x_{{4},5,2}, x_{{5},4,2}\big\}\\
&\circled{3}=\big
\{x_{1,1,2}, x_{2,4,2}, x_{3,2,1}, x_{{4},5,2}, x_{{5},3,2}\big\}          
&&\circled{4}=\big
\{x_{1,1,3}, x_{2,4,2}, x_{3,2,2}, x_{{4},5,2}, x_{{5},3,2}\big\}\\
&\circled{5}=\big
\{x_{1,2,1}, x_{2,1,2}, x_{3,5,1}, x_{{4},3,1}, x_{{5},4,2}\big\}          
&&\circled{6}=\big
\{x_{1,2,1}, x_{2,4,2}, x_{3,5,1}, x_{{4},1,2}, x_{{5},3,2}\big\}\\
&\circled{7}=\big
\{x_{1,3,1}, x_{2,1,2}, x_{3,2,1}, x_{{4},5,2}, x_{{5},4,2}\big\}          
&&\circled{8}=\big
\{x_{1,4,2}, x_{2,1,2}, x_{3,2,2}, x_{{4},5,2}, x_{{5},3,2}\big\}
\end{align*}

\begin{figure}
	\begin{center}
		\def\svgwidth{0.75\textwidth}
		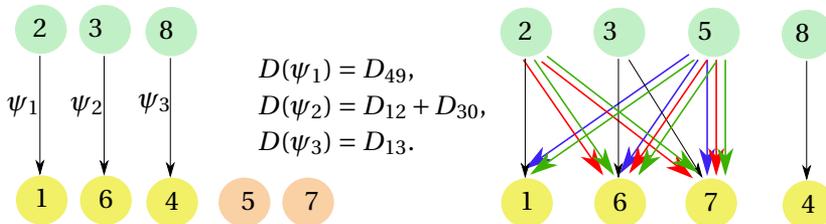	
		\caption{Left: The differential for $(\Sigma,\alphas,\betas,\mathbf{z_1})$, Right: The differential for  $(\Sigma,\alphas,\betas,z)$. The contributions from the disks $\phi_i$ for $i=1,\dots,5$, the disks $\phi_i-P_1$ for  $i=1,\dots,4$ and the disks $\phi_i+P_1$ for $i=1,2,5$ are denoted with green, red and blue arrows, respectively.}
\label{8-1}
\end{center}
\end{figure}

Let $\mathbf{z_1}$ consist of marked points in all domains except for  $D_{12}$, $D_{13}$, $D_{30}$, $D_{49}$. The differentials for the Heegaard diagram $(\Sigma,\alphas,\betas,\mathbf{z_1})$ along with the domains of the connecting disks are shown in Figure \ref{8-1} on the left. In this figure, a black arrow from a generator $x$ to a generator $y$ denotes that there is a disk from $x$ to $y$ with a unique holomorphic representative. In fact, the domains associated with all the disks are polygons. The group  $H_*(\widehat{CF}(\Sigma,\alphas,\betas,\mathbf{z_1}), \spinc_1)$ is thus isomorphic to $\Z_2^2$ and is generated by $C=\{\circled{5},\circled{7}\}$. To compute $\underline{\widehat{HF}}(M,\spinc_1)$, we  need to determine the matrices $l$, $n$, $p$, and $k$ in Lemma \ref{HF}. Define the disks  $\phi\in\pi_2\big(\circled{3},\circled{7}\big)$ and 
\[
\phi_1\in\pi_2\big(\circled{5},\circled{7}\big),\ \ 
\phi_2\in\pi_2\big(\circled{5},\circled{6}\big),\ \ 
\phi_3\in\pi_2\big(\circled{2},\circled{7}\big),\ \ 
\phi_4\in\pi_2\big(\circled{2},\circled{6}\big),\ \ 
\phi_5\in\pi_2\big(\circled{5},\circled{1}\big),
\]
of Maslov index $1$ by specifying their domains. If $I=\{4,16,18, 26,34,36,38,51,52\}$ and $D$ is the formal sum $\sum_{i\in I}D_i$, we have
\begin{align*}
	&D(\phi)=D_0,
	&&D(\phi_1)=\Sigma-\{D_{12}+D_{14}+D_{30}\},
	&&D(\phi_2)=\Sigma-\{D_{0}+D_{14}\},\\
	&D(\phi_3)=\Sigma-(D+D_{12}+D_{14}+D_{30}),
	&&D(\phi_4)=\Sigma-(D+D_0+D_{14}),
	&&D(\phi_5)=D+D_{49}.
\end{align*}

Then all the disks of index $1$ and positive domain between the generators of this complex  are  the disks $\psi_i$ for $i=1,2,3$, the disk $\phi$, the disks $\phi_i$ for $i=1,\dots,5$, the disks $\phi_i-P_1$ for  $i=1,\dots,4$ and the disks $\phi_i+P_1$ for $i=1,2,5$, see Figure \ref{8-1} (right). Let
\begin{align*}
	&b_i=\#\widehat{M}(\phi_i), &&\text{for}\ \ i=1,\dots,5,
	&&&&c_i=\#\widehat{M}(\phi_i-P_1),&&\text{for}\ \  i=1,\dots,4,\\
	&d_i=\#\widehat{M}(\phi_i+P_1),&&\text{for}\ \   i=1,2,
	&&\text{and}&&c_5=\#\widehat{M}(\phi_5+P_1).&&
\end{align*}

Setting $K=b_4+c_4e^{-P_1}$, $N_1=b_5+c_5e^{P_1}$,  $N_2=b_2+c_2e^{-P_1}+d_2e^{P_1}$, $P=b_3+c_3e^{-P_1}$ and $L=b_1+c_1e^{-P_1}+d_1e^{P_1}$, it then follows that 
\begin{align*}
k=\left(\begin{array}{ccc}
		0&0&0\\K&0&0\\0&0&0
	\end{array}
	\right),\quad
	n=\left(\begin{array}{cc}
		N_1&0\\N_2&0\\0&0
	\end{array}
	\right),\quad
	&p=\left(\begin{array}{ccc}
		0&0&0\\P&1&0
	\end{array}
	\right),\quad
	l=\left(\begin{array}{cc}
		0&0\\L&0
	\end{array}
	\right)\\
	\Rightarrow\quad &l+p(I+k)^{-1}n=\left(\begin{array}{cc}
		0&0\\L+N_1(P+K)+N_2&0
	\end{array}
	\right).
\end{align*}
Note that in this matrix we have
\begin{equation}
	\begin{split}\label{Eq-01}
		\star=L+N_1(P+K)+N_2=&(b_1+b_2+b_5b_3+b_5b_4+c_5c_3+c_5c_4)\\
		&+(c_1+c_2+b_5c_3+b_5c_4)e^{-P_1}+
		(d_1+d_2+c_5b_3+c_5b_4)e^{P_1}.
	\end{split}
\end{equation}
The computation of $\underline{\widehat{HF}}(M,\spinc_1)$ is thus reduced to a computation of $\star$. Consider the disks
\begin{align*}
&\lambda_1\in\pi_2\big(\circled{7},\circled{5}\big),
&&&&\lambda_2\in\pi_2\big(\circled{6},\circled{5}\big),
&&&&\lambda_3\in\pi_2\big(\circled{7},\circled{2}\big),\\
&\lambda_4\in\pi_2\big(\circled{6},\circled{2}\big),
&&&&\lambda_5\in\pi_2\big(\circled{1},\circled{5}\big),
&&\text{and}&&\lambda_6\in\pi_2\big(\circled{1},\circled{2}\big).
\end{align*}
which correspond to the domains
\begin{align*}
	&D(\lambda_i)=\Sigma-D(\phi_i),&&\text{for}\ \  i=1,\dots,5&&\text{and}&& D(\lambda_6)=\Sigma-D(\psi_1).
\end{align*}
The domains of $\lambda_1$ and $\lambda_2$ are polygons. Then all the positive disks of index $1$ in $\pi_2(\circled{1},x)$, $\pi_2(x,\circled{5})$, and $\pi_2(x,\circled{2})$, with $x$  a generator of $\underline{\widehat{CF}}(M,\spinc_1)$, which the region containing the marked point $z$ are $\lambda_i$ for $i=1,\dots,6$, $\lambda_i+P_1$ for $i=3,4,6$, and $\lambda_i-P_1$ for $i=5,6$. Let
\begin{align*}
&b'_i=\#\widehat{M}(\lambda_i),&&\text{for}\ \  i=3,\dots,6,
&&&&c'_i=\#\widehat{M}(\lambda_i+P_1),&&\text{for}\ \  i=3,4,\\
&c'_i=\#\widehat{M}(\lambda_i-P_1),&&\text{for}\ \   i=5,6,
&&\text{and}&&d'_6=\#\widehat{M}(\lambda_6+P_1).
\end{align*}

Consider the Whitney disk classes of index $2$ 
\[\eta_1,\eta'_1,\eta''_1\in\pi_2(\circled{5},\circled{5}),\quad\eta_2,\eta'_2,\eta''_2\in\pi_2(\circled{1},\circled{1})\quad\text{and}\eta_3,\eta'_3,\eta''_3\in\pi_2(\circled{2},\circled{2})\]  which correspond to the periodic domains 
\[D(\eta_i)=\Sigma,\quad D(\eta'_i)=\Sigma-P_1\quad\text{and}\quad D(\eta''_i)=\Sigma+P_1,\quad\quad\text{for}\ \ i=1,2,3.\] 
The possible degenerations of  $\eta_1,\eta_1$ and $\eta_1'$ to positive disks of Maslov index $1$ are:
\begin{align*}
&\eta_1=\phi_j*\lambda_j=(\phi_5+P_1)*(\lambda_5-P_1),&&\text{for}\ \ j=1,2,5,\\
&\eta_1'=(\phi_i-P_1)*\lambda_i=\phi_5*(\lambda_5-P_1),&&\text{for}\ \ i=1,2\quad\quad\text{and}\\
&\eta_1''=(\phi_i+P_1)*\lambda_i=(\phi_5+P_1)*\lambda_5,&&\text{for}\ \ i=1,2
\end{align*}
Therefore, the following three equations follow:
\begin{equation}\label{Eq-02}
		b_1+b_2+b'_5b_5+c'_5c_5=0,\quad\quad
		c_1+c_2+c'_5b_5=0\quad\quad\text{and}\quad\quad
		d_1+d_2+b'_5c_5=0.
\end{equation}

The possible degenerations of $\eta_2,\eta_2'$ and $\eta_2''$ into positive disks of Maslov index $1$ are 
\begin{align*}
&\eta_2=\lambda_6*\psi_1=\lambda_5*\phi_5=(\lambda_5-P_1)*(\phi_5+P_1)&&\\
&\eta'_2=(\lambda_6-P_1)*\psi_1=(\lambda_5-P_1)*\phi_5&&\text{and}\\
&\eta''_2=(\lambda_6+P_1)*\psi_1=\lambda_5*(\phi_5+P_1).&&
\end{align*}
Therefore, we obtain the following $3$ equations
\begin{equation}\label{Eq-03}
b'_6+b'_5b_5+c'_5c_5=0,\quad\quad
c'_6+c'_5b_5=0\quad\quad\text{and}\quad\quad
d'_6+b'_5c_5=0.
\end{equation}
Similarly, the possible degeneration of $\eta_3,\eta_3'$ and $\eta_3''$ into positive disks of Maslov index $1$ are
\begin{align*}
&\eta_3=\phi_i*\lambda_i=(\phi_i-P_1)*(\lambda_i+P_1)=\psi_1*\lambda_6&&\\
&\eta'_3=(\phi_i-P_1)*\lambda_i=\psi_1*(\lambda_6-P_1)&&\text{and}\\
&\eta''_3=\phi_i*(\lambda_i+P_1)=\psi_1*(\lambda_6+P_1)&&\text{for}\ \ i=3,4.
\end{align*}
 Therefore, we obtain the following $3$ equations as well
\begin{equation}\label{Eq-04}
b'_3b_3+b'_4b_4+c'_3c_3+c'_4c_4+b'_6=0,\quad
b'_3c_3+b'_4c_4+c'_6=0\quad\text{and}\quad
c'_3b_3+c'_4b_4+d'_6=0.
\end{equation}

 Let $\mathbf{z'_1}$ contain a marked point in all the regions of $\Sigma-\alphas-\betas$ except for those appearing in $D(\lambda_3)$, $D(\lambda_4)$, $D(\phi_5)$, and $D_{13}$ and $\partial_1$ denote the corresponding differential. Note that $P_1$, $P_2$ and $P_3-\Sigma$ may still be considered as a basis for the space of periodic domains. Therefore, the diagram remains admissible for this choice of marked points. Then
\begin{equation}\label{Eq-05}
\partial_1^2\circled{3}=(b'_3+b'_4)\circled{2}\quad\quad\text{and}\quad\quad \partial_1^2\circled{7}=(b_5+b'_3)\circled{1}\quad\quad\Rightarrow\quad\quad b'_3=b'_4=b_5.
\end{equation}

 Similarly, let $\mathbf{z'_2}$ contain a marked point in all the regions of $\Sigma-\alphas-\betas$ except for those appearing in $D(\lambda_3+P_1)$, $D(\lambda_4+P_1)$, $D(\phi_5+P_1)$, and $D_{13}$ and $\partial_2$ denote the corresponding differential. Then
\begin{equation}\label{Eq-06}
\partial_2^2\circled{3}=(c'_3+c'_4)\circled{2}\quad\quad\text{and}\quad\quad \partial_2^2\circled{7}=(c_5+c'_3)\circled{1}\quad\quad\Rightarrow\quad\quad c'_3=c'_4=c_5.
\end{equation}
If follows from Equations \ref{Eq-01}-\ref{Eq-06} that the matrix $l+p(I+k)^{-1}n=0$. Thus $\underline{\widehat{HF}}(M,\spinc_1)\neq0$.

\newpage
\section{The second non-trivial Heegaard-Floer group}\label{Computation-2}
Let $\spinc_2$ be the $\SpinC$ class which corresponds to  $(0,-1,-8)$. The chain complex $\underline{\widehat{CF}}(M,\spinc_2)$ is bigger, in comparison with $\underline{\widehat{CF}}(M,\spinc_1)$, and  is generated by the following $72$ generators:
{\scriptsize{
\begin{align*}
&\circled{1}=\{x_{1,1,2}, x_{2,2,4}, x_{3,4,1}, x_{4,5,1}, x_{5,3,2}\}          
&&\circled{2}=\{x_{1,1,3}, x_{2,2,5}, x_{3,4,1}, x_{4,5,1}, x_{5,3,2}\}
&&\circled{3}=\{x_{1,1,2}, x_{2,2,5}, x_{3,5,1}, x_{4,3,1}, x_{5,4,1}\}          
\\&\circled{4}=\{x_{1,1,2}, x_{2,3,2}, x_{3,2,2}, x_{4,5,1}, x_{5,4,1}\}
&&\circled{5}=\{x_{1,1,2}, x_{2,3,2}, x_{3,5,2}, x_{4,2,1}, x_{5,4,1}\}          
&&\circled{6}=\{x_{1,1,2}, x_{2,4,1}, x_{3,2,2}, x_{4,5,1}, x_{5,3,1}\}
\\&\circled{7}=\{x_{1,1,2}, x_{2,4,1}, x_{3,5,2}, x_{4,2,1}, x_{5,3,1}\}          
&&\circled{8}=\{x_{1,1,3}, x_{2,4,1}, x_{3,5,1}, x_{4,2,1}, x_{5,3,2}\}
&&\circled{9}=\{x_{1,2,2}, x_{2,3,2}, x_{3,1,1}, x_{4,5,1}, x_{5,4,2}\}          
\\&\circled{10}=\{x_{1,2,2}, x_{2,4,2}, x_{3,1,1}, x_{4,5,1}, x_{5,3,1}\}
&&\circled{11}=\{x_{1,3,1}, x_{2,1,1}, x_{3,2,1}, x_{4,4,7}, x_{5,5,1}\}         
&&\circled{12}=\{x_{1,3,1}, x_{2,1,1}, x_{3,2,1}, x_{4,4,7}, x_{5,5,2}\}
\\&\circled{13}=\{x_{1,3,1}, x_{2,1,1}, x_{3,2,2}, x_{4,4,5}, x_{5,5,1}\}         
&&\circled{14}=\{x_{1,3,1}, x_{2,1,1}, x_{3,2,2}, x_{4,4,5}, x_{5,5,2}\}
&&\circled{15}=\{x_{1,3,1}, x_{2,1,1}, x_{3,2,3}, x_{4,4,10}, x_{5,5,1}\}        
\\&\circled{16}=\{x_{1,3,1}, x_{2,1,1}, x_{3,2,3}, x_{4,4,10}, x_{5,5,2}\}
&&\circled{17}=\{x_{1,3,1}, x_{2,1,1}, x_{3,2,4}, x_{4,4,10}, x_{5,5,1}\}        
&&\circled{18}=\{x_{1,3,1}, x_{2,1,1}, x_{3,2,4}, x_{4,4,10}, x_{5,5,2}\}
\\&\circled{19}=\{x_{1,3,2}, x_{2,1,1}, x_{3,2,1}, x_{4,4,10}, x_{5,5,1}\}        
&&\circled{20}=\{x_{1,3,2}, x_{2,1,1}, x_{3,2,1}, x_{4,4,10}, x_{5,5,2}\}
&&\circled{21}=\{x_{1,3,2}, x_{2,1,1}, x_{3,2,2}, x_{4,4,8}, x_{5,5,1}\}         
\\&\circled{22}=\{x_{1,3,2}, x_{2,1,1}, x_{3,2,2}, x_{4,4,8}, x_{5,5,2}\}
&&\circled{23}=\{x_{1,3,1}, x_{2,2,3}, x_{3,1,2}, x_{4,5,1}, x_{5,4,2}\}         
&&\circled{24}=\{x_{1,3,2}, x_{2,2,3}, x_{3,1,1}, x_{4,5,1}, x_{5,4,2}\}
\\&\circled{25}=\{x_{1,3,1}, x_{2,2,1}, x_{3,1,1}, x_{4,4,7}, x_{5,5,1}\}         
&&\circled{26}=\{x_{1,3,1}, x_{2,2,1}, x_{3,1,1}, x_{4,4,7}, x_{5,5,2}\}
&&\circled{27}=\{x_{1,3,1}, x_{2,2,1}, x_{3,1,2}, x_{4,4,10}, x_{5,5,1}\}        
\\&\circled{28}=\{x_{1,3,1}, x_{2,2,1}, x_{3,1,2}, x_{4,4,10}, x_{5,5,2}\}
&&\circled{29}=\{x_{1,3,1}, x_{2,2,2}, x_{3,1,1}, x_{4,4,10}, x_{5,5,1}\}        
&&\circled{30}=\{x_{1,3,1}, x_{2,2,2}, x_{3,1,1}, x_{4,4,10}, x_{5,5,2}\}
\\&\circled{31}=\{x_{1,3,1}, x_{2,2,3}, x_{3,1,1}, x_{4,4,8}, x_{5,5,1}\}         
&&\circled{32}=\{x_{1,3,1}, x_{2,2,3}, x_{3,1,1}, x_{4,4,8}, x_{5,5,2}\}
&&\circled{33}=\{x_{1,3,1}, x_{2,2,4}, x_{3,1,1}, x_{4,4,6}, x_{5,5,1}\}         
\\&\circled{34}=\{x_{1,3,1}, x_{2,2,4}, x_{3,1,1}, x_{4,4,6}, x_{5,5,2}\}
&&\circled{35}=\{x_{1,3,1}, x_{2,2,4}, x_{3,1,2}, x_{4,4,9}, x_{5,5,1}\}         
&&\circled{36}=\{x_{1,3,1}, x_{2,2,4}, x_{3,1,2}, x_{4,4,9}, x_{5,5,2}\}
\\&\circled{37}=\{x_{1,3,1}, x_{2,2,5}, x_{3,1,1}, x_{4,4,4}, x_{5,5,1}\}         
&&\circled{38}=\{x_{1,3,1}, x_{2,2,5}, x_{3,1,1}, x_{4,4,4}, x_{5,5,2}\}
&&\circled{39}=\{x_{1,3,1}, x_{2,2,5}, x_{3,1,2}, x_{4,4,7}, x_{5,5,1}\}         
\\&\circled{40}=\{x_{1,3,1}, x_{2,2,5}, x_{3,1,2}, x_{4,4,7}, x_{5,5,2}\}
&&\circled{41}=\{x_{1,3,2}, x_{2,2,1}, x_{3,1,1}, x_{4,4,10}, x_{5,5,1}\}        
&&\circled{42}=\{x_{1,3,2}, x_{2,2,1}, x_{3,1,1}, x_{4,4,10}, x_{5,5,2}\}
\\&\circled{43}=\{x_{1,3,2}, x_{2,2,4}, x_{3,1,1}, x_{4,4,9}, x_{5,5,1}\}         
&&\circled{44}=\{x_{1,3,2}, x_{2,2,4}, x_{3,1,1}, x_{4,4,9}, x_{5,5,2}\}
&&\circled{45}=\{x_{1,3,2}, x_{2,2,5}, x_{3,1,1}, x_{4,4,7}, x_{5,5,1}\}         
\\&\circled{46}=\{x_{1,3,2}, x_{2,2,5}, x_{3,1,1}, x_{4,4,7}, x_{5,5,2}\}
&&\circled{47}=\{x_{1,3,2}, x_{2,2,5}, x_{3,1,2}, x_{4,4,10}, x_{5,5,1}\}        
&&\circled{48}=\{x_{1,3,2}, x_{2,2,5}, x_{3,1,2}, x_{4,4,10}, x_{5,5,2}\}
\\&\circled{49}=\{x_{1,3,1}, x_{2,2,3}, x_{3,4,2}, x_{4,1,1}, x_{5,5,1}\}         
&&\circled{50}=\{x_{1,3,1}, x_{2,2,3}, x_{3,4,2}, x_{4,1,1}, x_{5,5,2}\}
&&\circled{51}=\{x_{1,3,1}, x_{2,2,3}, x_{3,5,2}, x_{4,1,1}, x_{5,4,2}\}         
\\&\circled{52}=\{x_{1,3,1}, x_{2,2,5}, x_{3,5,1}, x_{4,1,2}, x_{5,4,1}\}
&&\circled{53}=\{x_{1,3,2}, x_{2,2,4}, x_{3,5,1}, x_{4,1,1}, x_{5,4,2}\}         
&&\circled{54}=\{x_{1,4,1}, x_{2,1,2}, x_{3,2,2}, x_{4,5,1}, x_{5,3,1}\}
\\&\circled{55}=\{x_{1,4,3}, x_{2,1,1}, x_{3,2,1}, x_{4,5,1}, x_{5,3,2}\}         
&&\circled{56}=\{x_{1,4,4}, x_{2,1,1}, x_{3,2,2}, x_{4,5,1}, x_{5,3,2}\}
&&\circled{57}=\{x_{1,4,1}, x_{2,1,2}, x_{3,5,2}, x_{4,2,1}, x_{5,3,1}\}         
\\&\circled{58}=\{x_{1,4,2}, x_{2,1,1}, x_{3,5,1}, x_{4,2,2}, x_{5,3,2}\}
&&\circled{59}=\{x_{1,4,4}, x_{2,1,1}, x_{3,5,1}, x_{4,2,1}, x_{5,3,1}\}         
&&\circled{60}=\{x_{1,4,4}, x_{2,1,1}, x_{3,5,2}, x_{4,2,1}, x_{5,3,2}\}
\\&\circled{61}=\{x_{1,4,2}, x_{2,2,3}, x_{3,1,1}, x_{4,5,1}, x_{5,3,1}\}         
&&\circled{62}=\{x_{1,4,2}, x_{2,2,4}, x_{3,1,2}, x_{4,5,1}, x_{5,3,2}\}
&&\circled{63}=\{x_{1,4,3}, x_{2,2,1}, x_{3,1,1}, x_{4,5,1}, x_{5,3,2}\}         
\\&\circled{64}=\{x_{1,4,3}, x_{2,2,4}, x_{3,1,1}, x_{4,5,1}, x_{5,3,1}\}
&&\circled{65}=\{x_{1,4,3}, x_{2,2,5}, x_{3,1,2}, x_{4,5,1}, x_{5,3,2}\}         
&&\circled{66}=\{x_{1,4,4}, x_{2,2,5}, x_{3,1,1}, x_{4,5,1}, x_{5,3,1}\}
\\&\circled{67}=\{x_{1,4,1}, x_{2,2,5}, x_{3,5,1}, x_{4,1,2}, x_{5,3,2}\}         
&&\circled{68}=\{x_{1,4,2}, x_{2,2,1}, x_{3,5,1}, x_{4,1,1}, x_{5,3,2}\}
&&\circled{69}=\{x_{1,4,2}, x_{2,2,4}, x_{3,5,1}, x_{4,1,1}, x_{5,3,1}\}         
\\&\circled{70}=\{x_{1,4,2}, x_{2,2,4}, x_{3,5,2}, x_{4,1,1}, x_{5,3,2}\}
&&\circled{71}=\{x_{1,4,3}, x_{2,2,5}, x_{3,5,1}, x_{4,1,1}, x_{5,3,1}\}         
&&\circled{72}=\{x_{1,4,3}, x_{2,2,5}, x_{3,5,2}, x_{4,1,1}, x_{5,3,2}\}
\end{align*}}}

\begin{figure}
	\begin{center}
		\def\svgwidth{0.9\textwidth}
		\fontsize{8}{10}\selectfont
		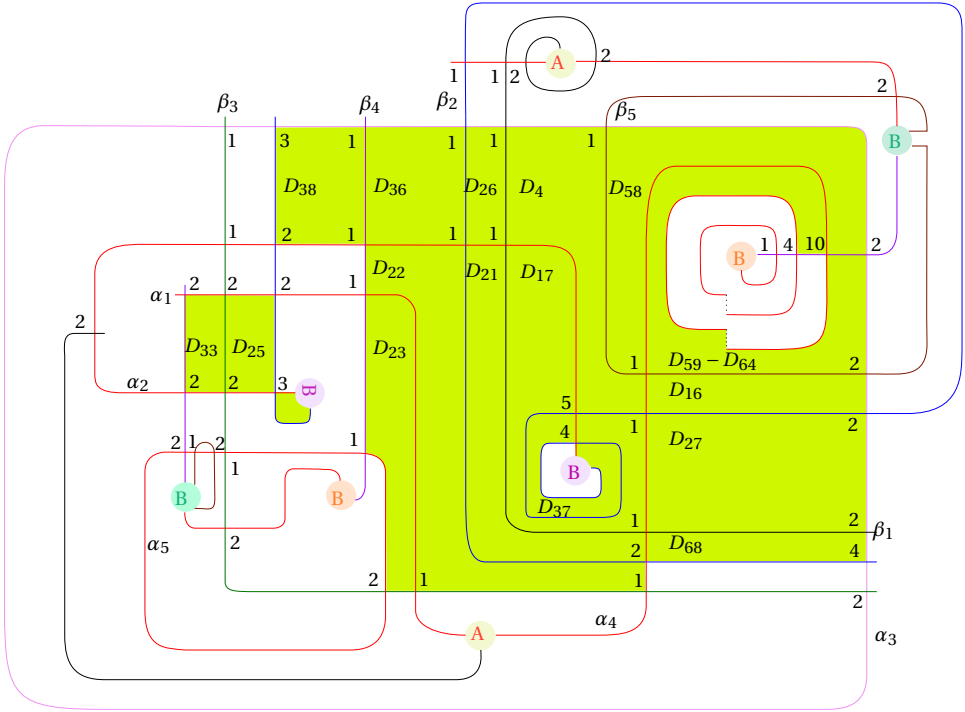	
	\caption{Part of the Heegaard diagram, where the marked points $\mathbf{z_1}$ are in the domains other than those in green.}
	\label{72-0}
\end{center}
\end{figure}

Let $\mathbf{z_1}$ consist of a marked point in all regions of the Heegaard diagram  except for $D_i$ with $i=4,16,17,21,22,23,25,26,27,36,37,38,58,59,60,61,62,63,64,68$. A neighborhood of these latter domains is illustrated  in Figure \ref{72-0}, where the aforementioned domains are colored green.\\

\begin{figure}
	\begin{center}
		\def\svgwidth{\textwidth}
		\fontsize{10}{10}\selectfont
		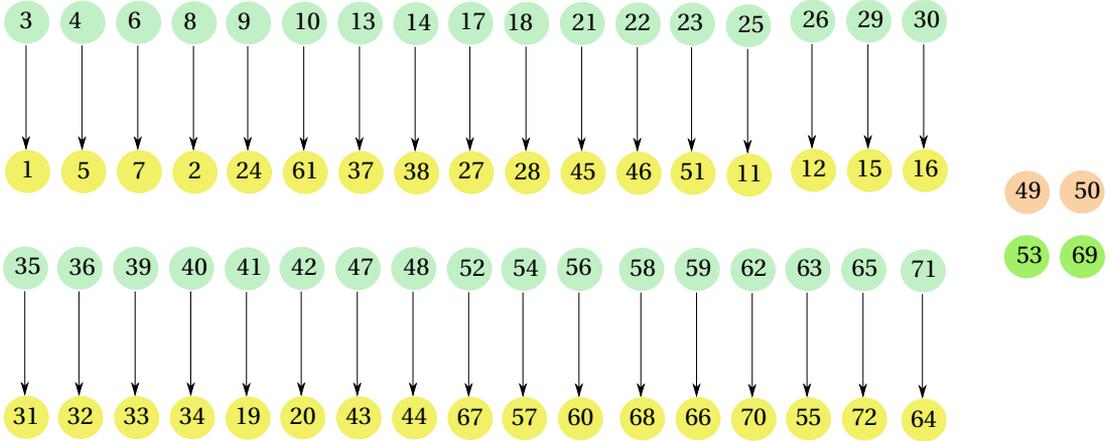	
	\caption{The differential corresponding to the diagram $(\Sigma,\alphas,\betas,\mathbf{z_1})$.}
	\label{72-2}
\end{center}
\end{figure}

The differential corresponding to the Heegaard diagram $(\Sigma,\alphas,\betas,\mathbf{z_1})$ is illustrated in Figure \ref{72-2}.  In fact, most of the positive Whitney disks of Maslov index $1$ for $(\Sigma,\alphas,\betas,\mathbf{z_1})$, which connect two of the aforementioned $72$ generators have polygonal domains, and their contribution to the differential is thus equal to $1$. There are precisely $12$ disks $\phi_i$ for $i=1,\dots,7$, and $\phi'_j$ for $j=1,2,5,6,7$, with non-polygonal domains, where we have
\begin{align*}
\phi_1&\in\pi_2\big(\circled{21},\circled{45}\big), & \phi'_1&\in\pi_2\big(\circled{22},\circled{46}\big), &
	\phi_2&\in\pi_2\big(\circled{13},\circled{37}\big),  & \phi'_2&\in\pi_2\big(\circled{14},\circled{38}\big),\\
	\phi_3&\in\pi_2\big(\circled{3},\circled{1}\big), &
	\phi_4&\in\pi_2\big(\circled{71},\circled{64}\big),&
	\phi_5&\in\pi_2\big(\circled{47},\circled{43}\big), & \phi'_5&\in\pi_2\big(\circled{48},\circled{44}\big),\\
	\phi_6&\in\pi_2\big(\circled{39},\circled{33}\big), & \phi'_6&\in\pi_2\big(\circled{40},\circled{34}\big), &
	\phi_7&\in\pi_2\big(\circled{35},\circled{31}\big), & \phi'_7&\in\pi_2\big(\circled{36},\circled{32}\big).
\end{align*}
The domains associated with these disks are 
{\small{\begin{align*}
&D(\phi_1)=D(\phi'_1)=D_4+D_{16}+D_{58}+D_{59}+D_{60}+D_{61},
&&D(\phi_2)=D(\phi'_2)=D(\phi_1)+D_{62}+D_{63}+D_{64},\\
&D(\phi_3)=D_4+D_{17}+D_{21}+D_{22}+D_{23}+D_{26}+D_{36},
&&D(\phi_4)=D_4+D_{17},\\
&D(\phi_5)=D(\phi'_5)=D_4+D_{16}+D_{17}+D_{27}+D_{58}+D_{59},
&&D(\phi_6)=D(\phi'_6)=D(\phi_5)+D_{60}+D_{61}+D_{62},\\
&D(\phi_7)=D(\phi'_7)=D(\phi_5)+D_{60}.
\end{align*}}}

By Lemma \ref{Non-Poly-1}, $\#\Mhat (\phi_i)=\#\Mhat (\phi'_j)=1$, for $i=1,\dots,4$ and $j=1,2$. Moreover, by Lemma \ref{Non-Poly-2}, $\#\Mhat (\phi_i)=\#\Mhat (\phi'_i)=1$ for $i=5,6,7$. Thus, the differential is as illustrated in Figure~\ref{72-2} and $H_*(\underline{\widehat{CF}}(\Sigma,\alphas,\betas,\mathbf{z_1}))$
 is generated by 
\[C=\big\{C_1=\circled{49},C_2=\circled{50},C_3=\circled{53},C_4=\circled{69}\big\}.\]

%
%
\begin{figure}
	\begin{center}
		\def\svgwidth{.8\textwidth}
		\fontsize{6}{10}\selectfont
		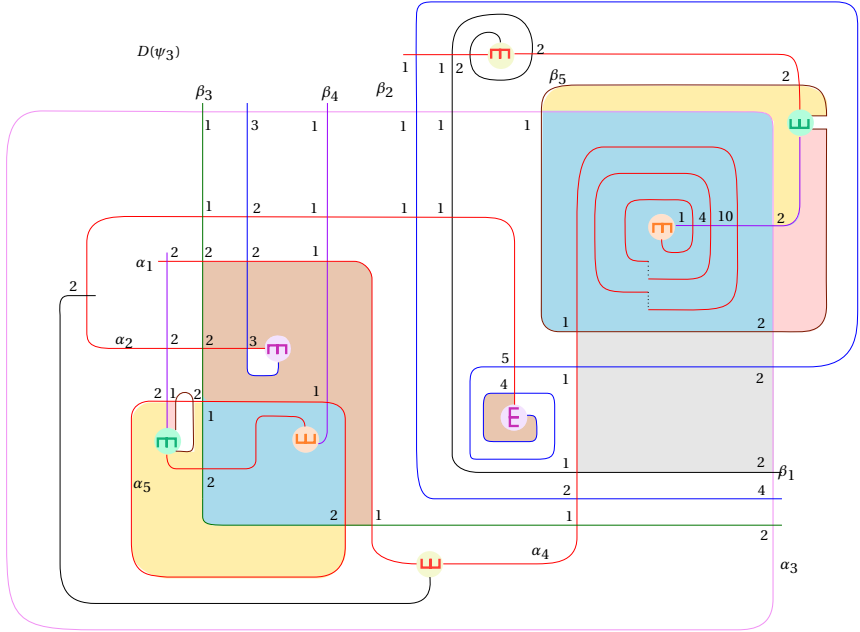	
		\caption{If $I_1,I_2,I_3,I_4$ and $I_5$ denote the set of indices of domains colored yellow, blue, pink, grey, and brown respectively, the domains associated with $\psi_1,\psi_3$ and $\psi_6$ are given by $D(\psi_1)=\sum_{i\in I_1\cup I_2\cup I_3} D_i$, $D(\psi_3)=\sum_{i\in I_2\cup I_3\cup I_5}D_i$ and $D(\psi_6)=\sum_{i\in I_2\cup I_4\cup I_5}D_i$.}
		\label{72-1-2}
	\end{center}
\end{figure}

To compute $\underline{\widehat{HF}}(M,\spinc_2)$, we  need to determine the matrices $l$, $n$, $p$, and $k$ in the Lemma \ref{HF}. 
All possible positive disks with Maslov index $1$ between the generators in $C$ are 
\[
\begin{aligned}
	\psi'_1,\psi_1&\in\pi_2(\circled{49},\circled{50}),\ \
	&\psi'_2,\psi_2\in\pi_2(\circled{53},\circled{69}),\\
	\psi_3&\in\pi_2(\circled{49},\circled{53}),\ \
	&\psi_4\in\pi_2(\circled{50},\circled{69}).
\end{aligned}
\]
For $i=1,2$, the domain associated with $\psi'_i$  is a polygon. The domain associated with $\psi_1$ is shown in Figure \ref{72-1-2} as the union of yellow, blue and pink regions.  By Lemma \ref{Non-Poly-3}, $\#\Mhat (\psi_1)=1$. The domain associated with $\psi_3$ is shown in Figure \ref{72-1-2} as the union of blue, brown and pink regions. Setting $V=\#\widehat{M}(\psi_3)$, we find
\[l=\left(\begin{array}{cccc}
	0&0&0&0\\
	1+e^{P_1}&0&0&0\\
	V&0&0&0\\
	0&*&*&0
\end{array}
\right).
\]

Let $A$ and $B$ denote the chain complexes generated by all the generators colored in light green and yellow in Figure \ref{72-2}, respectively. Denote the generators of $A$ and $B$ by $A_i$ and $B_i$, respectively. We may choose the labeling of the aforementioned generators of $A$ and $B$ such that $k$ is a lower triangular matrix. Therefore, 
\[
p(I+k)^{-1}n=pn+pkn+pk^2n+\dots. 
\]
For $j\geq 0$, since the coefficients are in $\Z_2$, each non-zero entry in $pk^jn$ is of the form 
\[
(pk^jn)_{wv}=\sum_{r_1,\dots,r_{j+1}}p_{wr_1}k_{r_1r_2}\dots k_{r_jr_{j+1}}n_{r_{j+1}v},
\] 
and implies the existence of positive disks $\lambda_t$ of Maslov index $1$ for $t=1,\dots,j+1$, where 
\begin{align*}
&\lambda_1\in\pi_2(C_v,B_{r_{j+1}}),
&&\lambda_{j+2}\in\pi_2(A_{r_1},C_w)&&\text{and}&&
&&\lambda_t\in\pi_2(A_{r_t},B_{r_{t-1}}),\ \ t=2,\dots,j+1.
\end{align*}
and $\#\widehat{M}(\lambda_t)=1$. In particular, $D(\lambda_t)>0$ for all $t$ and 
\begin{equation}\label{Eq-domain}
	D(\lambda_t)\subset\sum_{t=1}^{j+2}D(\lambda_t)=\sum_{t=1}^{j+1}D(\lambda'_t)+D(\lambda)\pm D(P_1),\ \ \mu(\lambda_t)=1, \ \ D(\lambda_t)>0,
\end{equation}
for some positive Whitney disks $\lambda'_t\in\pi_2(A_{r_t},B_{r_t})$  and $\lambda\in\pi_2(C_v,C_w)$ of Maslov index $1$. Potentially, there are only two such sequences satisfy (\ref{Eq-domain}), which are shown in Figure \ref{72-2-2}. Here $\psi_5\in\pi_2(\circled{49},\circled{51})$ is a disk with a polygonal domain. 
The domain associated with the disk $\psi_6\in\pi_2(\circled{23},\circled{53})$ is shown in Figure \ref{72-1-2} as the union of grey, brown and blue regions.\\ 

\begin{figure}
	\begin{center}
		\def\svgwidth{0.5\textwidth}
		\fontsize{10}{10}\selectfont
		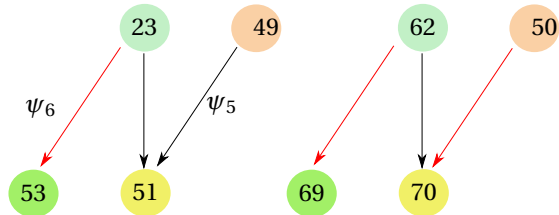	
		\caption{ Potential sequences of disks corresponding to non-zero summands $p_{wr_1}k_{r_1r_2}\dots k_{r_jr_{j+1}}n_{r_{j+1}v}$ in $(pk^jn)_{wv}$.}
		\label{72-2-2}
	\end{center}
\end{figure}

\begin{figure}
	\begin{center}
		\def\svgwidth{\textwidth}
		\fontsize{6}{10}\selectfont
		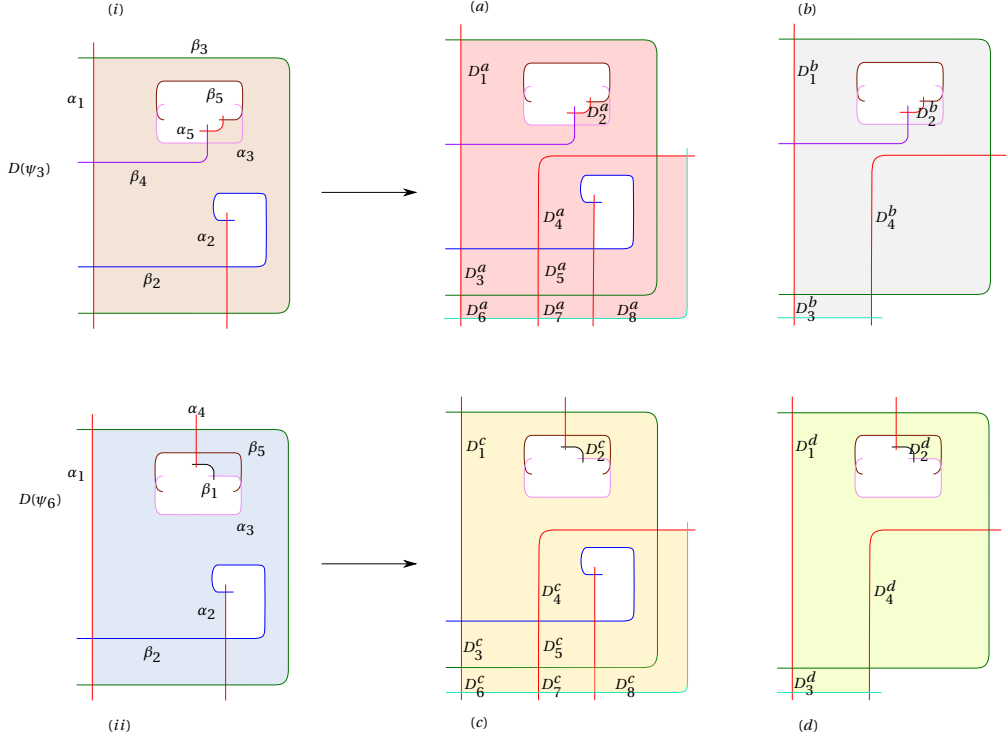	
	\caption{Domains associated with disks $\eta_a$, $\eta_b$, $\eta_c$, $\eta_d$, in $(a)$, $(b)$, $(c)$, $(d)$, respectively.}
		\label{72-6-1}
	\end{center}
\end{figure}

\begin{lemma}\label{lem-psi-3-6}
With the above notation in place, we have $\#\Mhat (\psi_3)=\#\Mhat (\psi_6)$.
\end{lemma}
\begin{proof}
The domains associated with $\psi_3$ and $\psi_6$ are extended in two ways in Figure~\ref{72-6-1}. The domains $D^\bullet_i$ for $\bullet=a,b,c$ and $d$ denote the components of the regions colored pink, grey, yellow and green, respectively. They determine the domains of Whitney disks $\eta_a,\eta_b,\eta_c$ and $\eta_d$, respectively, with domains
$D(\eta_\bullet)=\sum_{i}D^\bullet_i$. Note that for $\bullet=a,c$ the values of $i$ are in $\{1,\ldots,8\}$, while for $\bullet=b,d$ the values of  $i$ are in $\{1,\ldots,4\}$. The  possible degenerations for the disks $\eta_a,\eta_b,\eta_c$ and $\eta_d$ are given by
\[\eta_\bullet=\phi^\bullet_1*\sigma^\bullet_1=\sigma^\bullet_2*\phi^\bullet_2=\sigma^\bullet_3*\phi^\bullet_3\quad\quad\bullet=a,c,\]
\[\eta_\bullet=\phi^\bullet_1*\sigma^\bullet_1=\sigma^\bullet_2*\phi^\bullet_2=\phi^\bullet_3*\sigma^\bullet_3\quad\quad\bullet=b,d,\]
where the corresponding domains are given by
\begin{align*}
&D(\phi^\bullet_1)=\sum_{j=1}^5D^\bullet_j,
&&D(\phi^\bullet_2)=D^\bullet_1+\sum_{j=3}^8D^\bullet_j,
&&D(\phi^\bullet_3)=D^\bullet_1+D^\bullet_2+D^\bullet_3+D^\bullet_6\\
&D(\sigma^\bullet_3)=D^\bullet_4+D^\bullet_5+D^\bullet_7+D^\bullet_8,
&&D(\phi^{\star}_1)=D^{\star}_1+D^{\star}_2+D^{\star}_3,
&&D(\phi^{\star}_2)=D^{\star}_1+D^{\star}_3+D^{\star}_4,\\
&\text{and}&&D(\phi^{\star}_3)=D^{\star}_1+D^{\star}_2+D^{\star}_4
\end{align*}
for $\bullet=a,c$ and $\star=b,d$, while the domains associated with $\sigma^a_j,\sigma^c_j,\sigma^b_i$ and $\sigma^d_i$ are polygons for $j=1,2$ and $i=1,2,3$. By Lemma \ref{Non-Poly-1}, we also have $\#\widehat{M}(\sigma^\bullet_3)=1$, $\bullet=a,c$. Therefore
\begin{equation}\label{eqq-01}
	\sum_{i=1}^3\#\widehat{M}(\phi^\bullet_i)=0\quad\quad\text{for}\ \ \bullet=a,b,c,d.
\end{equation}

On the other hand, we have 
\begin{align*}
&D(\phi^a_1)=D(\psi_3), 
&&D(\phi^c_1)=D(\psi_6),
&&D(\phi^a_2)=D(\phi^c_2),
&&D(\phi^{b}_2)=D(\phi^{d}_2),\\
&D(\phi^{a}_3)=D(\phi^{b}_1),
&&D(\phi^{c}_3)=D(\phi^{d}_1),
&&\quad\quad\text{and}&&D(\phi^{b}_3)=D(\phi^{d}_3).
\end{align*}
Thus by \ref{eqq-01}, we have $V=\#\Mhat (\psi_3)=\#\Mhat (\psi_6)$.
\end{proof}

Having established Lemma~\ref{lem-psi-3-6}, we conclude that 
\[p(I+k)^{-1}n=\left(\begin{array}{cccc}
	0&0&0&0\\
	0&*&*&*\\
	V&*&*&*\\
	0&*&*&*
\end{array}
\right)\quad\quad\Rightarrow\quad\quad
l+p(I+k)^{-1}n=\left(\begin{array}{cccc}
	0&0&0&0\\
	1+e^{P_1}&*&*&*\\
	0&*&*&*\\
	0&*&*&*
\end{array}
\right).
\]
This means that $\circled{49}$ survives in  $\underline{\widehat{HF}}(M,\spinc_2)$, and the latter group is thus non-trivial.

\end{document}